\documentclass{article}

\input{amssym.def}
\input{amssym.tex}

\usepackage{amsthm}
\usepackage{epsfig}
\usepackage{color}

\begin{document}

\title{Determining Fuchsian groups by their finite quotients}
\author{M.R.~Bridson\thanks{supported by a Senior Fellowship from the
EPSRC of Great Britain and a Wolfson Research Merit Award from the 
Royal Society of London},~M.D.E.~Conder\thanks{supported in part by the 
N.Z.~Marsden Fund (grant UOA1015) and a James Cook Fellowship of the Royal Society of 
New Zealand}~ \&~ A.W.~Reid\thanks{supported in part 
by NSF grants}}

\def\ker{\rm{ker}}
\def\sup{\rm{sup}}
\def\inf{\rm{inf}}
\def\ab{\rm{ab}}
\def\cd{{\rm{cd}}}
\def\mod{{\rm mod}}
\def\lcm{{\rm lcm}}
\def\gcd{{\rm gcd}}
\def\PSL{{\rm PSL}}
\def\SL{{\rm SL}}
\def\SO{{\rm SO}}
\def\SU{{\rm SU}}
\def\GL{{\rm GL}}
\def\PGL{{\rm PGL}}
\def\Isom{\rm{Isom}}
\def\qed{ $\sqcup\!\!\!\!\sqcap$}
\def\rank{\mbox{\rm{rank}}}
\def\Ad{\mbox{\rm{Ad}}}
\def\tr{\mbox{\rm{tr}}}
\def\P{\mbox{\rm{P}}}
\def\O{\mbox{\rm{O}}}
\def\G{\Gamma}
\def\Z{\Bbb Z}
\def\C{\mathcal{C}}
\def\Hom{{\rm{Hom}}}
\def\Epi{{\rm{Epi}}}
\def\conj{\frak{cf}} 
\def\<{\langle}
\def\>{\rangle}
\def\wh{\widehat}
\def\R{\Bbb R}
\def\CC{\Bbb C}
\def\Q{\Bbb Q}
\def\ds{\displaystyle}

\newtheorem{theorem}{Theorem}[section]
\newtheorem{lemma}[theorem]{Lemma}
\newtheorem{corollary}[theorem]{Corollary}
\newtheorem{proposition}[theorem]{Proposition}
\newtheorem{prop}[theorem]{Proposition}



\theoremstyle{definition} 
\newtheorem{definition}[theorem]{Definition}
\newtheorem{remark}[theorem]{Remark}
\newtheorem{question}[theorem]{Question}

\maketitle

\begin{abstract}
  Let $\C(\G)$ be the set of isomorphism classes of the finite groups
  that are quotients (homomorphic images) of $\G$. We investigate the extent to which
  $\C(\G)$ determines $\G$ when $\G$ is a group of geometric interest.
  If $\Gamma_1$ is a lattice in ${\rm{PSL}}(2,\R)$ and $\Gamma_2$ is a
  lattice in any connected Lie group, then ${\cal C}(\Gamma_1) = {\cal
    C}(\Gamma_2)$ implies that $\Gamma_1 \cong \Gamma_2$. If $F$ is a free
  group and $\G$ is a right-angled Artin group or a residually free
  group (with one extra condition), then $\C(F)=\C(\G)$ implies that 
  $F\cong\G$. If $\G_1<{\rm{PSL}}(2,\Bbb C)$ and $\G_2< G$ are
  non-uniform arithmetic lattices, where $G$ is a semi-simple Lie
  group with trivial centre and no compact factors, then $\C(\G_1)=
  \C(\G_2)$ implies that $G \cong {\rm{PSL}}(2,\Bbb C)$ and that $\G_2$ belongs to one
  of finitely many commensurability classes. These results are proved
  using the theory of profinite groups; we do not exhibit explicit
  finite quotients that distinguish among the groups in question. But in 
  the special case of two non-isomorphic triangle groups, 
  we give an explicit description of finite quotients that distinguish
  between them.
\end{abstract}

%
%
%
%
\section{Introduction}

Let $\G$ be a finitely-generated group and let $\C(\G)$ denote the set of isomorphism classes of finite groups
that are quotients (homomorphic images) of $\G$.
If  $\G$ is residually finite, then one can
recover any finite portion of its Cayley graph or multiplication table
by examining the finite quotients of the group.  It
is therefore
natural to wonder whether, under reasonable hypotheses, the set 
$\C(\G)$ might determine $\Gamma$ up to isomorphism. 
(One certainly needs some hypotheses: for example, Remeslenikov  \cite{rem} 
showed that 
a finitely-generated nilpotent group is not always uniquely determined by $\C(\G)$.)

A celebrated instance of this question is the following (see
Problem (F14) in  \cite{NYG}):
If $F_n$ is the free group of rank $n$, and $\Gamma$ is a finitely-generated, 
residually finite group, then does
${\cal C}(\Gamma) = {\cal C}(F_n)$ imply that $\Gamma\cong F_n$? This question
remains out of reach for the moment, as does the broader question of whether every Fuchsian group 
is distinguished from other finitely-generated, residually finite groups by its set of finite quotients.
But in  this paper we 
shall answer these questions in the affirmative for groups $\Gamma$ that 
belong to various classes of groups that cluster naturally around Fuchsian groups. 
For example, we prove the following:

\begin{theorem}
\label{main1}
Let $\Gamma_1$ be a finitely-generated Fuchsian group and
let $\Gamma_2$ be a lattice in a connected Lie group.
If ${\cal C}(\Gamma_1) =
{\cal C}(\Gamma_2)$, then 
$\Gamma_1 \cong \Gamma_2$.
\end{theorem}

We remind the reader that a Fuchsian group is, by definition, an infinite 
 discrete subgroup of
$\PSL(2,{\R})$. To avoid trivial special cases, we shall assume that all of the groups considered
are non-elementary (that is,~are not virtually cyclic).
Thus, for us, every Fuchsian group $\G$ has a subgroup of finite index that maps onto a
non-abelian free group, and hence every finitely-generated group 
is a quotient of some finite-index subgroup of $\G$.  Deciding which groups 
 arise as quotients
of $\Gamma$ itself is a more subtle matter, but much progress has been made
on understanding the finite quotients, that is,~deciphering the structure of ${\cal C}(\Gamma)$; 
see \cite{Con}, \cite{Ev}, \cite{LS}
and \cite{Mac} and references therein.      

The structure of the set $\C(\Gamma)$ is intimately connected with the subgroup growth
of $\Gamma$, so it is interesting to contrast our results with what is known about the
subgroup growth of Fuchsian groups (see \cite{MSP} and
\cite{LS}). Let $s_n(\Gamma)$ denote the
number of subgroups in $\G$ that have index precisely $n$.  
In \cite{MSP}
an equivalence relation on the set of finitely-generated groups was introduced:
 $\Gamma_1$ and $\Gamma_2$ are declared to be equivalent if and only if 
$s_n(\Gamma_1) = (1 + o(1))s_n(\Gamma_2)$ asymptotically (that is, as $n\rightarrow \infty).$
Also in \cite{MSP}, M\"uller and Schlage-Puchta exhibited an infinite sequence of pairwise 
non-isomorphic Fuchsian
groups that are all equivalent in this sense.  Our Theorem~\ref{main1}
shows that these equivalent groups are distinguished by their finite quotients.

Here, and throughout this paper, we use the term `surface group' to mean 
a group that is isomorphic to the fundamental group of a closed surface of 
genus at least $1$.
A basic case in Theorem~\ref{main1} is the situation where $\Gamma_1$ is a free group 
and $\Gamma_2$ is a surface group. This is not a difficult case to handle, but we give several proofs,
each pointing to an argument that enables one to distinguish free groups by means of $\C(\Gamma)$ in
broader classes; Theorem \ref{t:others} records some of these classes and  broader, more technically
defined classes will be described in Section~\ref{s:FvS}. 
One proof relies on the observation that
in a surface group, every subgroup of finite index has even first betti number,
whereas any finitely-generated free group has a subgroup of finite index with
odd first betti number.  
A second proof relies on the observation that the number of
elements needed to generate a surface group is equal to the rank of its abelianisation.
 A third proof  relies on the fact that surface groups  
are {\em{good}} in the sense of Serre (see
Section~\ref{s:good}), LERF (locally extended residually finite), 
and of cohomological dimension 2 over a finite field.

\begin{theorem}\label{t:others}
 Let $\G_1$ be a free group, and let $\G_2$ be a finitely-generated group.  
Then $\C(\G_1)\neq \C(\G_2)$ if $\,\G_2$ satisfies one of the following conditions$\,:$
\\[-16pt]  
\begin{enumerate}
\item[{\rm (a)}] $\G_2$ is the fundamental group of a compact K\"ahler manifold$\,;$ \\[-18pt]  
\item[{\rm (b)}]  $\G_2$ is residually free and contains a surface subgroup$\,;$ \\[-18pt]  
\item[{\rm (c)}]  $\G_2$ is a non-free, right-angled Artin group.
\end{enumerate}
\end{theorem}

In the preceding discussion, we have regarded finitely-generated free
groups as examples of Fuchsian groups, residually free groups,
right-angled Artin groups, and lattices in connected Lie groups. But
except in the case of Fuchsian groups, we have not addressed the
question of whether the operator $\C$ distinguishes the isomorphism types
within these classes. For general lattices, it certainly does not: it
has long been known that there exist lattices in distinct nilpotent
Lie groups that have the same collection of finite quotients
\cite{segal} Chapter 11, and Aka \cite{A1} and \cite{A2} recently
provided examples of such pairs of lattices in distinct semisimple Lie groups $G_1, G_2$. Indeed there
are examples where the real ranks of $G_1$ and $G_2$ are
different (for example, rank 2 and rank 10) but  nevertheless
there exist lattices $\G_j<G_j$ such that $\C(\G_1)=\C(\G_2)$.  There also
exist pairs of finitely-generated residually free groups
$\G_1\not\cong\G_2$ such that $\C(\G_1)=\C(\G_2)$; see Remark
\ref{r:PvFxF}.
In addition, in the context of 3-manifold groups, Funar \cite{Fu}
has shown the existence of non-homeomorphic
torus bundles over the circle whose fundamental groups have the same finite
quotients.

In contrast, there are reasons to suspect that for lattices in
$\PSL(2,\CC)$, and in particular the fundamental groups of finite-volume 
orientable hyperbolic 3-manifolds, the situation may be more
reminiscent of ${\rm{PSL}}(2,\R)$, with lattices being uniquely
determined by their finite quotients, both amongst themselves and
amongst lattices in arbitrary Lie groups.  In pursuit of this
conviction, we prove the following theorem in Section~\ref{s:psl2C}.
 
\begin{theorem}
\label{mainbianchi}
Let $\Gamma_1$ be a non-uniform lattice in $\PSL(2,{\CC})$, 
and let $\Gamma_2$ be a non-uniform irreducible
arithmetic lattice in a semisimple Lie group $G$ that has
trivial centre and no compact factors. If ${\cal C}(\Gamma_1) =
{\cal C}(\Gamma_2)$ then $G\cong \PSL(2,{\CC})$. Moreover, if $\G_1$
is arithmetic then the family of all $\G_2$ with ${\cal C}(\Gamma_1) =
{\cal C}(\Gamma_2)$ divides into
finitely many commensurability classes.
\end{theorem}

It is interesting to compare the conclusion of Theorem~\ref{mainbianchi}
with Pickel's Theorem \cite{Pi} and recent work by Aka \cite{A2}. 
In the latter it is shown that among higher rank lattices which have the Congruence
Subgroup Property, there are only finitely many 
that can have the same set of finite quotients, whilst in the former it is shown 
that only finitely many finitely-generated nilpotent
groups can have the same set of finite quotients. 

\medskip

We have deliberately stated our results in terms of the naively defined set
$\C(\G)$, but it is both natural and useful to regard them as statements about
the {\em profinite completions} of the groups in question. We remind the reader that
the profinite completion $\widehat{\Gamma}$ of a group $\G$ is the inverse limit of the finite quotients
$\G/N$ of $\G$. (The maps in the inverse system are the obvious ones: if $N_1 <N_2$
then $\G/N_1\rightarrow \G/N_2$.) If $\G$ is finitely-generated,
then $\C(\G)=\C(\widehat{\Gamma})$; 
thus $\wh{\G}_1\cong\wh{\G}_2$ implies $\C({\G_1})=\C({\G_2})$. 
Less obviously,
$\C(\G)$ uniquely determines $\widehat{\Gamma}$  (see \cite{DFPR}, 
\cite{RZ} p.89 and Section 3 below). Thus
we may rephrase all of the preceding results in terms of profinite completions. For example: 

\begin{theorem}
\label{profinite}
Let $\Gamma_1$ and $\Gamma_2$ be as in  Theorem~{\em\ref{main1}}  
or~{\em\ref{t:others}}.  If $\widehat{\Gamma}_1 \cong \widehat{\Gamma}_2$
then  $\Gamma_1\cong \Gamma_2$.\end{theorem} 

Our invocation of profinite groups is more than a matter of terminology: 
in particular, our proofs of Theorems~\ref{main1} and \ref{mainbianchi} rely
on various aspects of the theory of profinite groups, and the interplay
between the abstract group and its profinite completion. Certain of
the arguments also rely on deep properties of the lattices concerned.
 
Our detours through the theory of profinite groups are such that, for the most part,
we prove that groups have distinct profinite completions {\em without exhibiting specific
finite quotients that distinguish them}. Indeed, even in relatively small examples, identifying
such finite quotients appears to be rather delicate. 
In the final section of this paper we tackle this problem head-on in the case of triangle groups,
where it succumbs to a direct analysis. 
\medskip

We close this introduction with some comments concerning logic and decidability.
First, note  that whenever one
has a class of finitely-presented residually finite groups where $\widehat{\Gamma}_1
\cong \widehat{\Gamma}_2$ implies $\G_1\cong\G_2$, one 
obtains a solution to the isomorphism 
problem in that class of groups. Indeed, since each group is finitely-presented, one
can effectively enumerate its finite quotients, and in this way one will prove in a finite 
number of steps that $\C(\G_1)\neq\C(\G_2)$ if this is the case; and running this partial
algorithm in parallel with a naive search for mutually-inverse homomorphisms $\G_1\leftrightarrow\G_2$ 
will (in finite time) determine whether or not $\G_1$ is isomorphic to $\G_2$.
Similarly, Theorem~\ref{main1} implies that there exists an algorithm that, given a finite presentation
of a lattice in a connected Lie group, will determine whether or not the group presented is
Fuchsian. And Theorem~\ref{mainbianchi} provides an algorithm for determining if a finite presentation
of a non-uniform arithmetic lattice is Kleinian.

A recent result of Jarden and Lubotzky \cite{JL}, which relies on deep
work of Nikolov and Segal \cite{NS}, shows that there exists an isomorphism
$\widehat{\Gamma}_1\cong\widehat{\Gamma}_2$ if and only if there is an
equivalence of the elementary theories ${\rm{Th}}(\widehat\G_1)\equiv
{\rm{Th}}(\widehat\G_2)$ (in the sense of first order logic). So our results
can also be interpreted as criteria for determining if such elementary theories are
distinct. This is particularly striking in the light of the solution
to Tarski's problem by Sela \cite{sela} and Kharlampovich-Myasnikov
\cite{KM}, which implies that the fundamental group of any closed
surface of Euler characteristic at most $-2$ has the same elementary
theory as a free group of any finite rank $r\ge 2$.  For clarity, we
highlight a special case of this discussion, but stress that we have
contributed only a small part of the proof.

\begin{theorem} If $\G_1$ and $\G_2$ are any two non-isomorphic, non-elementary,
 torsion-free Fuchsian groups, then 
${\rm{Th}}(\G_1)\equiv {\rm{Th}}(\G_2)$ but  ${\rm{Th}}(\widehat{\Gamma}_1)\not\equiv {\rm{Th}}(\widehat{\Gamma}_2)$.
\end{theorem}
   
We conclude this introduction with a brief outline of the paper. In Section 2 we recall some
basic properties of the profinite completion of a (discrete)
residually finite group $\G$, the correspondence between subgroups of
$\G$ and that of its profinite completion, as well as the relationship
between the profinite completion of subgroups and their closures in
the profinite completion of $\G$. In addition, we recall some of the
theory of the cohomology of profinite groups, including Serre's notion
of goodness and cohomological dimension, and we discuss
actions of profinite groups on profinite trees. In Section 3 we describe
a connection (via L\"uck's Approximation Theorem) between groups with
isomorphic profinite completions and the first $L^2$-Betti number. In
Section 4 we prove a variety of results that restrict the class of groups that
can have the same profinite completion as a free group, and in Section 5 we
prove a result controlling torsion in the profinite completion of a
finitely-generated Fuchsian group. The proof of Theorem~\ref{main1} is
completed in Section 6, by putting together results proved in Section 2, 3 and 5. 
In Section 7 we prove Theorem~\ref{mainbianchi}, relying heavily on Section 2.  
In Section 8 we explain how to distinguish triangle groups more directly, 
by describing exhibit explicit finite quotients that do so.
\\[\baselineskip]
\noindent{\bf Acknowledgements:}~{\em We thank Gopal Prasad and Alex Lubotzky for 
helpful comments concerning Theorem 3.6. The first and third authors
thank the University of Auckland, Massey University and The New
Zealand Institute for Advanced Study for supporting the visit during which
this collaboration began. They particularly thank Gaven Martin for his
hospitality during that visit. They also thank the Mittag-Leffler Institute for its hospitality
during the final drafting of this paper.}

\section{Profinite Completions}
 
In this section we recall some background on profinite completions and
the theory of profinite groups; see \cite{RZ}, \cite{segal} and \cite{Se}
for more details. 

\subsection{Residually finite groups and their completions} 

By definition, the {\em profinite completion} $\widehat{\Gamma}$ of a finitely-generated group $\G$ is the
inverse limit of the system of finite quotients of $\G$. The natural
map $\G\to\widehat{\Gamma}$ is injective if and only if $\G$ is residually finite.
By definition $\widehat{\Gamma}$ maps onto every finite quotient of $\G$. On the
other hand, 
the image of $\G$ is dense regardless of whether $\G$ is residually
finite, so the restriction to $\G$ of any continuous homomorphism from $\widehat{\Gamma}$
to a discrete group (such as a finite group) is surjective. A deep theorem of Nikolov and Segal
\cite{NS} states that every homomorphism from $\widehat \G$ 
to a finite group is continuous,
so we have the following basic result (in which $\Hom(G,Q)$ denotes the set of 
homomorphisms from the group $G$ to the group $Q$, and $\Epi(G,Q)$ denotes the set of epimorphisms).

\begin{lemma}\label{NS} Let $\G$ be a finitely-generated group and let  $\iota:\G\to\widehat{\Gamma}$
be the natural map to its profinite completion. Then, for every finite group $Q$,
the map $\Hom(\widehat{\Gamma},Q)\to \Hom(\G,Q)$ defined by $g\mapsto g\circ\iota$ is a bijection,
and this restricts to a bijection $\Epi(\widehat{\Gamma},Q)\to \Epi(\G,Q)$. 
\end{lemma}

\begin{corollary}\label{c:sameHom} If $\G_1$ is finitely-generated and 
$\widehat{\Gamma}_1\cong\widehat{\Gamma}_2$,  then $|\Hom(\G_1,Q)|=|\Hom(\G_2,Q)|$ for every finite group $Q$.
\end{corollary}

Closely related to this we have the following proposition (which is taken from \cite[Chapter 3.2]{RZ},
using the Nikolov-Segal theorem to replace `open' 
by `finite-index'). \\[\baselineskip]
\noindent{\bf{Notation.}}  Given a subset $X$ of a pro-finite group
$G$, we write $\overline X$ to denote the closure of $X$ in $G$.

\begin{proposition}
\label{correspondence}
With $\Gamma$ as above, there is a one-to-one correspondence between the
set ${\cal X}$ of finite-index subgroups of $\Gamma$ 
and the set ${\cal Y}$ of all finite-index subgroups
of $\widehat{\Gamma}$.  If $\G$ is identified with its image in $\widehat{\Gamma}$,
then this bijection takes $H\in {\cal X}$ to $\overline{H}$, 
and conversely, takes $Y\in {\cal Y}$ to $Y \cap \Gamma$.
Also $|\Gamma : H| = |\widehat{\Gamma}: \overline{H}|$. 
Moreover, $\overline H$ is normal in $\widehat\G$ if and only if $H$ is normal in $\G$, 
in which case $\G/H\cong\wh{\G}/\overline H$.
\end{proposition}
 
\begin{corollary}\label{c:exhaust} Let $\G$ be a finitely-generated group, and for each $d\in\Bbb N$, let
$M_d$ denote the intersection of all normal subgroups of index at most $d$ in $\G$. Then the 
closure $\overline M_d$ of $M_d$ in $\widehat \G$ is the intersection of all normal subgroups
of index at most $d$ in $\widehat{\Gamma}$, and hence $\bigcap_{d \in \Bbb N} \overline M_d = 1$.
\end{corollary}

\begin{proof} If $N_1$ and $N_2$ are the kernels of epimorphisms from $\G$ to finite groups $Q_1$
and $Q_2$, then $\overline{N_1\cap N_2}$ is the kernel of the extension of $\G\to Q_1\times Q_2$
to $\widehat{\Gamma}$, while $\overline N_1\times \overline N_2$ is the kernel of the map $\widehat{\Gamma}\to Q_1\times Q_2$
that one gets by extending each of $\G\to Q_i$ and then taking the direct product. The uniqueness of
extensions tells us that these maps coincide, and hence $\overline{N_1\cap N_2} =\overline N_1
\cap \overline N_2$. The claims follow from repeated application of this observation.
\end{proof} 

We finish this subsection with a discussion of the relationship between
the statement
$\C(\G_1)=\C(\G_2)$ for finitely-generated residually finite
groups $\Gamma_1$ and $\Gamma_2$,  
and the statement that the profinite 
completions $\widehat{\Gamma}_1$ and $\widehat{\Gamma}_2$ are isomorphic.

Suppose that $\widehat{\Gamma}_1$ and $\widehat{\Gamma}_2$ are isomorphic.
Following the aforementioned work of Nikolov and Segal \cite{NS},
by `isomorphic' we simply mean `isomorphic as groups', for by
\cite{NS}, each group-theoretic
isomorphism is continuous, and so any such isomorphism will be
an isomorphism of topological groups.
Given an isomorphism between $\widehat{\Gamma}_1$
and $\widehat{\Gamma}_2$, it
is easy to deduce from Proposition~\ref{correspondence} that this implies
that $\Gamma_1$ and $\Gamma_2$ have the same collection of finite quotient
groups. The converse is also true; see \cite{DFPR} and \cite[pp.~88--89]{RZ}.

\begin{theorem}
\label{allfinite}
Suppose that $\Gamma_1$ and $\Gamma_2$ are finitely-generated abstract
groups. Then $\widehat{\Gamma}_1$ and $\widehat{\Gamma}_2$ are isomorphic if
and only if ${\cal C}(\Gamma_1) = {\cal C}(\Gamma_2)$.
\end{theorem}
 
We close the subsection with a comment on terminology.
In \cite{GZ} (see also \cite{A2}) the notion of genus of a group
$\G$ from within a class of groups $\cal G$ is discussed: it is
defined to be the set
of isomorphism classes of groups in $\cal G$ with the same profinite
completion as $\Gamma$, and is denoted by ${\bf g}({\cal G},\G)$.  Hence 
if we let $\cal L$ denote the set of all lattices in connected Lie groups, 
our main result can be rephrased as:

\begin{theorem}
\label{genusmain}
Let $\G$ be a finitely-generated Fuchsian group, then 
${\bf g}({\cal L},\G)=\{\G\}$.
\end{theorem}

We have chosen not to use this notation elsewhere in this article in order to avoid confusion with the
classical use of the term genus in the sense of surfaces (and their fundamental groups).

\subsection{Completion and closure}  

Let $\Gamma$ be a residually finite (abstract) group,
and $u:H\hookrightarrow \Gamma$ the inclusion mapping of a proper subgroup
$H$ of $\Gamma$.  The canonical homomorphism $\widehat{u}:\widehat{H}
\rightarrow \widehat{\Gamma}$ is injective if and only if {\em{the profinite topology 
on $\Gamma$ induces the full profinite topology on $H$}}; in more elementary terms, for
every normal subgroup $H_1<H$ of finite index,
there is a normal subgroup $\Gamma_1 <\Gamma$ of finite index with
$\Gamma_1\cap H < H_1$. 

Recall that if $\G$ is a group and $H$ a subgroup of $\G$, then $\G$ is
called $H$-separable if for every $g\in G\smallsetminus H$, there is a
subgroup $K$ of finite index in $\G$ such that $H \subset K$ but
$g\notin K$; equivalenly, the intersection of
all finite index subgroups in $\Gamma$ containing $H$ is precisely $H$.
The group $\G$ is called {\em LERF} (or {\em subgroup separable}) if it is
$H$-separable for every finitely-generated subgroup $H$, or equivalently, 
if every finitely-generated subgroup is a closed subset in the profinite topology. 

In the context of this paper, it is important
to note that even if the subgroup $H$ of $\G$ is separable, it need not be the case that the profinite topology 
on $\Gamma$ induces the full profinite topology on $H$.
Stronger separability properties do suffice, however, as we now indicate. 
The following lemma  is well-known, but we include a proof for the convenience
of the reader.

\begin{lemma}
\label{inducefulltop}
Let $\Gamma$ be a finitely-generated group, and $H$ a finitely-generated subgroup of $\Gamma$. 
Suppose that $\Gamma$ is $H_1$-separable for every finite index subgroup $H_1$ in $H$. 
Then the profinite topology on $\Gamma$ induces the full profinite topology on $H$; that is,
the natural map $\widehat{H} \rightarrow \overline{H}$ is an 
isomorphism. \end{lemma}

\begin{proof} Since $\Gamma$ is $H_1$ separable, 
the intersection of all subgroups of finite index in $\Gamma$ containing $H_1$ 
is $H_1$ itself.  Thus we can find a finite-index subgroup $K_1$ of $\Gamma$ such that $K_1\cap H = H_1$. Replacing $K_1$
by a normal subgroup $\Gamma_1$ of finite index provides the required
subgroup $\Gamma_1\cap H$ of $H_1$.
\end{proof}

Subgroups of finite index are obviously separable.

\begin{corollary}\label{c1}
If $\G$ is residually finite and $H$ is a finite-index subgroup of $\Gamma$, then 
the natural map from $\widehat{H}$ to $\overline H$ (the closure of $H$ in $\widehat{\Gamma}$) is an isomorphism. 
\end{corollary} 

\begin{corollary}\label{same}
If $\G_1$ and $\G_2$ are finitely-generated and residually finite and $\widehat{\Gamma}_1\cong\widehat{\Gamma}_2$,
then there is a bijection $I$ from the set of subgroups of finite index in $\G_1$ to
the set of subgroups of finite-index in $\G_2$, such that $\widehat{I(H)}\cong\widehat{H}$
and $|\G_1:H|=|\G_2:I(H)|$ for every subgroup $H$ of finite index in $\G_1$.
Moreover, $H$ is normal in $\G_1$ if and only if $I(H)$ is normal in $\G_2$.
\end{corollary}

\begin{proof} If $\phi:\widehat{\Gamma}_1\to\widehat{\Gamma}_2$ is an isomorphism,  define $I(H)= \phi(\overline H)\cap\G_2$.
Proposition~\ref{correspondence} tells us that $I$ is a bijection and that
$|\G_1:H|=|\G_2:I(H)|$. Corollary~\ref{c1}
tells us that $\widehat{H}\cong \overline H \cong \overline{I(H)} \cong \widehat{I(H)}$.
\end{proof}

\subsection{Betti numbers}

The first betti number of a finitely-generated group is
$$b_1(\Gamma) = \hbox{rank}((\Gamma/[\Gamma,\Gamma]) \otimes_{\Z} {\Q}).$$
This invariant can be detected in the finite quotients of $\G$ since
it is the {\em least integer $b$ such that $\G$ has the elementary $p$-group of rank $b$ 
among its quotients, for every prime $p$.} We exploit this observation as follows:

\begin{lemma}\label{l:denseb1}
Let $\Lambda$ and $\G$ be finitely-generated groups. If $\Lambda$ is isomorphic to a dense
subgroup of $\widehat{\G}$, then $b_1(\Lambda) \ge b_1(\Gamma)$.
\end{lemma}

\begin{proof} For every finite group $A$, each epimorphism $\widehat{\Gamma}\to A$
will restrict to an epimorphism from both $\Gamma$ and $\Lambda$, and every epimorphism
$\Gamma\to A$ extends to an epimorphism $\widehat{\Gamma}\to A$. 
Therefore, if $\G$ maps onto an elementary $p$-group of rank $b$, 
then so does $\Lambda$ (but perhaps not vice versa).
\end{proof}

\begin{corollary}\label{sameb1} If $\G_1$ and $\G_2$ are finitely-generated 
and $\widehat{\Gamma}_1\cong\widehat{\Gamma}_2$,
then $b_1(\G_1)=b_1(\G_2)$.
\end{corollary}

\subsection{Goodness and cohomological dimension}\label{s:good} 

Following Serre \cite{Se}, we say that a group $\G$ is {\em good} 
if for every finite $\G$-module $M$, the homomorphism of cohomology groups
$$H^n(\widehat{\G};M)\rightarrow H^n(\G;M)$$
induced by the natural map $\G\rightarrow \widehat{\G}$ is an isomorphism
between the cohomology of $\G$ and the continuous cohomology of $\widehat{\G}$. 
(See \cite{Se} and \cite[Chapter 6]{RZ} for details about the cohomology of profinite groups.) 

It is easy to see that free groups are good, but it seems that 
goodness is hard to establish in general. One can, however,
establish goodness for a group $\G$ that is LERF if one has a well-controlled
splitting of the group as a graph of groups \cite{GJZ}. 
Using the fact that Fuchsian groups are LERF \cite{Sc},
it is proved in \cite{GJZ} that Fuchsian groups are good.

Now a useful criterion for goodness is provided by the next lemma due to 
Serre (see \cite[Chapter 1, Section 2.6]{Se})

\begin{lemma}
\label{goodexact}
The group $\G$ is good if there is a short exact sequence 
$$ 1\rightarrow N\rightarrow \G\rightarrow H\rightarrow 1,$$
such that $H$ and $N$ are good, $N$ is finitely-generated, and the
cohomology group $H^q(N,M)$ is finite for every $q$ and every finite
$\G$-module $M$.
\end{lemma}

Recent work of Agol \cite{agol} and Wise \cite{Wis} establishes that
finite-volume hyperbolic 3-manifolds are virtually fibered, and so
Lemma \ref{goodexact} can be applied to establish goodness for the
fundamental groups of such manifolds.  In addition, since goodness is
preserved by commensurability (see \cite{GJZ} Lemma 3.2), 
it follows that the fundamental groups of all finite-volume hyperbolic
3-manifolds are good. We
summarize this discussion, and that for Fuchsian groups, in the
following result:

\begin{theorem}
\label{good}
Lattices in $\PSL(2,{\Bbb R})$ and $\PSL(2,{\Bbb C})$ are good.
\end{theorem}

\medskip

Let $G$ be a profinite group.  
Then the {\em $p$-cohomological dimension} of $G$ is the least integer $n$ such 
that for every discrete torsion $G$-module $A$ and for 
every $q>n$, the $p$-primary component of $H^q(G;A)$ is zero, 
and this is  denoted by $\cd_p(G)$. 
The cohomological dimension of $G$ is defined 
as the supremum of $\cd_p(G)$ over all primes $p$, 
and this is denoted by $\cd(G)$.  

We also retain the standard notation
$\cd(\G)$ for the cohomological dimension (over $\Z$) of a 
{\em discrete} group $\G$.
%

\begin{lemma} 
\label{cdim}
Let $\G$ be a discrete group that is good. If $\cd(\G)\le n$,
then $\cd(\widehat{\Gamma})\le n$.
\end{lemma}


Discrete groups of finite cohomological dimension (over $\Z$) are torsion-free. We are interested in conditions that allow one to deduce that $\wh{\G}$ is also torsion-free.
For this we need the following result
that mirrors the behaviour of cohomological
dimension for discrete  abstract groups (see \cite[Chapter 1 \S 3.3]{Se}).

\begin{proposition}
\label{cdbound}
Let $p$ be a prime, let $G$ be a profinite group, and $H$ a closed subgroup of $G$. 
Then $\cd_p(H) \leq \cd_p(G)$.
\end{proposition}


\begin{corollary}
\label{notorsion2}
Suppose that $\Gamma$ is a residually finite, good group  
of finite cohomological dimension over $\Z$. Then $\widehat{\G}$ is
torsion-free.
\end{corollary}

\begin{proof} If $\wh{\G}$ were not torsion-free, then it would have 
an element $x$ of prime order, say $q$. Since  $\<x\>$ is a closed
subgroup of $\widehat{\G}$,  Proposition~\ref{cdbound} tells us that
$\cd_p(\<x\>)\leq \cd_p(\widehat{\G})$ for all primes $p$.  
But $H^{2k}(\<x\>;{\bf F}_q)\neq 0$ for all $k>0$, so
$\cd_q(\<x\>)$ and $\cd_q(\wh{\G})$ are infinite. 
Since $\G$ is good and has finite cohomological dimension over $\Z$, 
we obtain a contradiction from Lemma~\ref{cdim}.
\end{proof}

The following simple observation will prove very useful in the sequel. 

\begin{corollary}
\label{cdboundapply}
Let 
$\Gamma_1$ and $\Gamma_2$ be finitely-generated (abstract)
residually finite groups with $\widehat{\Gamma}_1 \cong \widehat{\Gamma}_2$. Assume
that $\Gamma_1$ is good and ${\cd}(\Gamma_1) = n<\infty$. Furthermore, assume
that $H$ is a good
subgroup of $\Gamma_2$ for which the natural mapping 
$\widehat{H}\rightarrow \widehat{\Gamma}_2$ is injective. 
Then $H^q(H;{\bf F}_p)=0$ for all $q > n$ and all primes $p$.
\end{corollary}

\begin{proof} If $H^q(H;{\bf F}_p)$ were non-zero for some
$q>n$, then by goodness we would have  $H^q(\widehat{H};{\bf F}_p)\neq
0$, so $\cd_p(\widehat{H}) \geq q>n$. Now
$\widehat{H}\to \widehat{\Gamma}_2$ is injective and so 
$\widehat{H}\cong\overline{H}$.  Hence $\widehat{\Gamma}_1$ contains a
closed subgroup of $p$-cohomological dimension greater than $n$, a 
contradiction.
\end{proof}

\subsection{Actions on profinite trees} 

In this subsection we recall some basics of Bass-Serre
theory for profinite groups that will be required later. 
(See \cite{MZ} and \cite{RZ} for more detailed accounts.) 

If $G_1$ and $G_2$ are profinite groups with a common profinite subgroup $H$, 
then we denote by $G_1 \amalg_H G_2$ the {\em profinite amalgamated free product}  
of $G_1$ and $G_2$, with $H$ amalgamated.

As in the case of abstract groups, if a profinite group $G$
splits as a profinite amalgamated free product, 
then there is control on where any given finite subgroup lies. The following
result was proved for free profinite products in \cite{HeR} and for more
general profinite amalgamated free products in \cite[Theorem 3.10]{MZ}.

\begin{theorem}
\label{finite_in_vertex}
Let $G=G_1 \amalg_H G_2$ be any profinite amalgamated free product. 
If $K$ is a finite subgroup of $G$, then $K$ is conjugate to a subgroup of $G_1$ or $G_2$.
\end{theorem}

Now if $A$ and $B$ are abstract groups with a common subgroup $C$, 
and $\Gamma$ denotes the abstract amalgamated free product $A *_C B$, 
then it is not generally true that $\widehat{\Gamma}$ is isomorphic to
$\widehat{A} \amalg_{\widehat{C}} \widehat{B}$.  
But this does hold when $\Gamma$ is residually finite and the
profinite topology on $\Gamma$ induces the full profinite topologies
on $A$, $B$ and $C$; see \cite[pp.~379--380]{RZ}.  
In particular, we record the following consequence that will be useful for us.

\begin{proposition}
\label{profinitesplitting}
If $\Gamma = \G_1 *_C \G_2$ is LERF, then 
$\widehat{\Gamma} \cong \widehat{\Gamma}_1 \amalg_{\widehat{C}} \widehat{\Gamma}_2$.
\end{proposition}

A case of  particular  interest to us occurs when $\Gamma$ is a 
finitely-generated Fuchsian group; such groups are known to be LERF (see \cite{Sc}). 

\begin{corollary}
\label{profinitesurface}
If $\Gamma =A *_C B$ is a finitely-generated Fuchsian group, then
$\widehat{\Gamma} \cong \widehat{A} \amalg_{\widehat{C}}
\widehat{B}$.
\end{corollary}

Note that it follows from \cite{LMR} that if $\Gamma$ is any finitely-generated
Fuchsian group that is not a $(p,q,r)$-triangle group, then $\Gamma$
can be decomposed as non-trivial free products with amalgamation $A *_{C} B$. 
(For most Fuchsian groups, this follows easily using a
geometric decomposition of the hyperbolic 2-orbifold.)

\section{$L^2$-Betti numbers and profinite completions}\label{s:L2}

$L^2$-Betti numbers provide powerful invariants for distinguishing the profinite
completions of various classes of groups that cluster around free groups. We refer the
reader to L\"uck's treatise \cite{Lu2} for a comprehensive introduction to $L^2$-Betti numbers.
For our purposes, it is best to view these invariants not in terms of their original (more
analytic) definition, but instead as asymptotic invariants of towers of finite-index subgroups.
This is made possible by the L\"uck's Approximation Theorem \cite{Lu1}:

\begin{theorem}
\label{luck} Let $G$ be a finitely-presented group, 
and let 
$G = G_1 > G_2 > \ldots > G_m > \ldots$ be a sequence of finite-index subgroups
that are normal in $G$ and  intersect in the identity. The first $L^2$-Betti
number of $G$ is given by the formula
$$\displaystyle b_1^{(2)}(G) = \lim\limits_{m\rightarrow \infty} 
        {{{b_1(G_m)}\over{|G:G_m|}}}.$$
\end{theorem}

An important point to note is that this limit does not depend on the tower, and
hence is an invariant of $G$.

\begin{proposition} 
\label{dense}
Let $\Lambda$ and $\G$ be finitely-presented residually
finite groups
and suppose that $\Lambda$ is a dense subgroup of  $\widehat{\Gamma}$. 
Then  $b_1^{(2)}(\Gamma)\le  b_1^{(2)}(\Lambda)$.
\end{proposition}

\begin{proof} For each positive integer $d$ let $M_d$ be the
  intersection of all normal subgroups of index at most $d$ in $\G$,
  and let $L_d = \Lambda\cap\overline M_d$ in $\widehat{\Gamma}$.  We saw
  in Corollary~\ref{c:exhaust} that $\bigcap_d \overline M_d = 1$,
  and so $\bigcap_d L_d=1$.  Since $\Lambda$ and $\G$ are both dense in
  $\widehat{\Gamma}$, the restriction of $\widehat{\Gamma}\to
  \widehat{\Gamma}/\overline M_d$ to each of these subgroups is
  surjective, and hence
$$|\Lambda : L_d| = |\widehat{\Gamma}:\overline M_d| = |\G : M_d|.$$

Now $L_d$ is dense in $\overline M_d$, while $\widehat{M}_d = \overline M_d$,
so Lemma~\ref{l:denseb1} implies that $b_1(L_d) \ge b_1(M_d)$, 
and then we can use the towers $(L_d)$ in $\Lambda$ and $(M_d)$ in $\G$ to compare
$L^2$-betti numbers and find 
$$
b_1^{(2)}(\G) = \lim\limits_{d\rightarrow \infty}  \frac{b_1(M_d)}{|\G:M_d|} \le  
\lim\limits_{d\rightarrow \infty} \frac{b_1(L_d)}{|\Lambda:L_d|} = b_1^{(2)}(\Lambda), 
$$
by L\"uck's approximation theorem. 
\end{proof}
 
This has the following important consequence:

\begin{corollary}\label{c:sameL2}
Let $\G_1$ and $\G_2$ be finitely-presented residually finite 
groups. If $\widehat{\Gamma}_1\cong\widehat{\Gamma}_2$, then 
$b_1^{(2)}(\G_1) = b_1^{(2)}(\G_2) $.
\end{corollary}
 
\begin{remark} In \cite{BRe} the first and third authors prove
versions
of Proposition \ref{dense} and Corollary \ref{c:sameL2} for finitely-presented
groups which are residually-$p$ for some prime $p$.  In this
setting the profinite completion is replaced by the pro-$p$
completion.\end{remark}
When we make further use of $L^2$-betti numbers, we will exploit the following additional 
elementary observation: 

\begin{proposition}\label{p:b1pos}
If $\G$ is a lattice in ${\rm{PSL}}(2,\R)$ with rational Euler characteristic
$\chi(\G)$, then $b_1^{(2)}(\G) = -\chi(\G)$.
\end{proposition}

\begin{proof} It follows from L\"uck's approximation theorem that  if $H$ is a subgroup of index 
index $d$ in $\G$ (which is finitely-presented) then $b_1^{(2)}(H)=d\ b_1^{(2)}(\G)$. 
Rational Euler characteristic is multiplicative in  the same sense. 
Thus we may pass to a torsion-free subgroup of finite index in $\G$, 
and assume that it is either a free group $F_r$ of rank $r$, 
or the fundamental group $\Sigma_g$ of a closed orientable surface of genus $g$. 
In both cases, a subgroup $\G_d$ of index $d$ in $\G$ will be a group of the same form,
with the rank and genus given by the expressions $d(r-1) +1$ and $d(g-1)+1$. 
In the free case, the (ordinary) first betti number is $d(r-1) +1$ and so $b_1(\G_d) = 1- d\ \chi(\G)$, 
while in the second case the first betti number is $2d(g-1)+2$ and so $b_1(\G_d)= 2 - d\ \chi(\G)$.
In both cases, dividing by $d=|\G:\G_d|$ and taking the limit, we find $b_1^{(2)}(\G) = -\chi(\G)$. 
\end{proof}

We shall now present a reduction of Theorem~\ref{main1} using Corollary~\ref{c:sameL2} 
and Proposition~\ref{p:b1pos}.  The proof of Theorem~\ref{main1} will be completed in Sections 5 and 6.

We draw the reader's attention to the fact  that in Corollary~\ref{c:sameL2}
both groups are assumed to be residually finite, whereas in the following theorem
$\Lambda$ need not be residually finite since we have not assumed
that the Lie group $G$ is linear.  The proof 
shows how to by-pass this problem.

\begin{theorem}
\label{patch}
Let $\Lambda$ be a lattice in a connected Lie group $G$ and suppose that
$\wh\Lambda\cong\wh\G$ for some Fuchsian group $\G$. Then, either $\Lambda$ is
isomorphic to a Fuchsian group, or else $\wh{\Lambda}$ has a 
non-trivial finite normal subgroup.
\end{theorem}

\begin{proof} 
  Let $Z(G)$ denote the centre of $G$ and let $Z=\Lambda\cap Z(G)$. As
  $Z(G)$ is the kernel of the adjoint representation, $G/Z(G)$ is
  linear. In addition, since $\Lambda$ is a lattice in a connected Lie group
it is finitely generated (see \cite{rag} \S 6.18), 
and it follows that $\Lambda/Z$ is a finitely generated residually
finite group. In particular, $Z$ contains the kernel of the natural map
  $\Lambda\to\wh\Lambda$, which we call $I$. We first claim that either $Z=I$ or
  else $\Delta:=\Lambda/I$ has a non-trivial
finite normal subgroup; at the end of this
proof we shall show that such a subgroup remains normal in
  $\wh\Delta$, which is isomorphic to $\wh\Lambda$ because $\Delta$ is the
  image of $\Lambda$ in $\wh\Lambda$.  The only other possibility is that $Z/I$
  is an infinite central subgroup of the residually finite group
  $\Delta$. But $\Delta$ cannot have an infinite normal amenable
  subgroup, because Cheeger-Gromov \cite{CG} prove that this would 
  force  the first
  $L^2$-betti number of $\Delta$ to be zero, contradicting
Corollary \ref{c:sameL2} and Proposition \ref{p:b1pos}, 
since $\wh\Delta\cong\wh\G$.

Now we deal with the case $I=Z$, adapting an argument of Lott
\cite{Lo}.  Let ${\rm{Rad}}$ be the radical of $G$, let $K$ be the
maximal compact connected normal subgroup of a Levi complement, let
$G_1 = {\rm{Rad}}\cdot K$ and recall that $G_2 = G/G_1$ is a connected
semisimple Lie group with no compact connected normal subgroup.  The
short exact sequence
 $$
 1\to G_1\to G\to G_2\to 1
 $$
 restricts to 
 $$
 1\to \Lambda_1\to \Lambda\to \Lambda_2\to 1,
 $$
 where $\Lambda_1=\Lambda\cap G_1$, and factoring out  $Z$ (and its image $Z'$ in $\Lambda_2$) we get
  $$
 1\to \Lambda_1Z/Z\to \Lambda/Z\to \Lambda_2/Z'\to 1.
 $$
 By repeating the argument with $L^2$-betti numbers, we see that $\Lambda_1Z/Z$ must be finite,
as it is amenable (since $\Lambda_1$ is a closed subgroup of the amenable $G_1)$
 and  $\Delta = \Lambda/Z$ has the same profinite completion as the
 Fuchsian group $\Gamma$. So either we again obtain a finite
 non-trivial normal subgroup in $\Delta$, or else $\Lambda_1Z/Z$ is
 trivial, in which case $\Delta\cong \Lambda_2/Z'$. Now we appeal to the
 fact that since $\Lambda_2$ is a lattice in a {\em semisimple} Lie group
 with no compact factors, $\Lambda_2$ modulo its centre $Z(\Lambda_2)$ is itself
 a lattice in a semisimple Lie group (see \cite{rag} Corollary
 8.27). If $Z'=Z(\Lambda_2)$, then $\Delta = \Lambda_2/Z(\Lambda_2)$ and we are in
 the setting considered by Lott (\cite{Lo}, Lemma 1), who proves
 that either $\Delta$ has a non-trivial finite normal subgroup or else
 $\Delta$ is isomorphic to a Fuchsian group. If $Z'$ is a proper
 subgroup of $Z(\Lambda_2)$ then $\Delta$ has a non-trivial centre and we
 conclude as before.
 
 \medskip
 
 All that remains is to prove that if $M<\Delta$ is a non-trivial
 normal subgroup, then $M$ remains normal in $\wh\Delta$.  Indeed for
 any group $H$, if $N$ is a finite normal subgroup of $H$, then the
 image of $N$ in $\wh{H}$, which we denote $\overline{N}$, is also normal.
 For if it were not normal then there would be
 $n\in \overline{N}$ and $x\in \wh{H}$ such that 
 the set $S=\{xnx^{-1}n' \mid n'\in\overline{ N}\}$ did not contain the identity.
 By residual finiteness, the finite set $S\cup\{1\}$  would inject into some
 finite quotient of $\widehat{H}$. But $H$ maps onto this finite quotient, and $N$ is normal
 in $H$. This contradiction shows that $\overline{N}$  is normal in
 $\wh{H}$, as claimed.\end{proof}

\begin{remark} In Section 5 we shall conclude that the second possibility
in Theorem \ref{patch} does not occur, thereby reducing the proof of 
Theorem \ref{main1} to Fuchsian groups.\end{remark}

\section{Three Obstructions to Profinite Freeness}\label{s:FvS}

In this section we present three different proofs of the fact that the
fundamental group of a closed surface cannot have the same set of finite quotients
as a free group. 
Each proof highlights a  different obstruction to having the same profinite completion as a free group,
and thereby sheds light on the fundamental question of whether a 
finitely-generated residually finite group that has the same profinite completion as a free
group must itself be free.

\subsection{The Hopf Property and Rank}

It is well-known and easy to prove that if a finitely-generated group $\G$ 
is residually finite, then it has the Hopf property
--- that is, every epimorphism $\G\to\G$ is an isomorphism. Certain
proofs of this fact extend in a straightforward manner to finitely-generated profinite
groups. The following lemma captures this idea in more pedestrian language.

\begin{lemma}\label{l:hopf} Let $\phi:\G_1\to\G_2$ be an epimorphism of finitely-generated 
groups. If $\G_1$ is residually finite and $\widehat{\Gamma}_1\cong\widehat{\Gamma}_2$, then
$\phi$ is an isomorphism.
\end{lemma}

\begin{proof} Let $k\in\ker\,\phi$. If $k$ were non-trivial, then since $\G_1$ is residually
finite, there would be a finite group $Q$ and an epimorphism $f:\G_1\to Q$ such that $f(k)\neq 1$.
This map $f$ does not lie in the image of the injection $\Hom(\G_2,Q)\hookrightarrow
\Hom(\G_1,Q)$ defined by $g\mapsto g\circ\phi$. Thus $|\Hom(\G_1,Q)| > |\Hom(\G_2,Q)|$,
contradicting Corollary~\ref{c:sameHom}.
\end{proof}

\begin{definition} The {\em{rank}} $d(\G)$ of a finitely-generated group $\G$ is the least integer
 $k$ such that $\G$ has a generating set of cardinality $k$. The {\em{rank}} $\widehat d(G)$ of a 
 profinite group $G$ is the  least integer $k$ for which there is a
 subset $S\subset G$ with $k=|S|$ and $\<S\>$ is dense in $G$.
\end{definition}

In the following proposition, we do not assume that $\G$ is residually finite.

\begin{proposition}\label{p:epic} Let $\G$ be a finitely-generated group and let $F_n$
be a free group. If $\G$ has a finite
quotient $Q$ such that $d(\G)=d(Q)$, and $\widehat{\Gamma}\cong\widehat F_n$, then $\G\cong F_n$.
\end{proposition}

\begin{proof} 
First $\widehat{\Gamma}\cong\widehat F_n$, so $Q$ is a quotient of $F_n$. Hence $n\ge d(Q)$.
But $d(Q)=d(\G)$ and for every integer $s\ge d(\G)$ there exists an epimorphism $F_s\to \G$.
Thus we obtain an epimorphism $F_n\to\G$, and application of the preceding lemma completes the proof.
\end{proof} 

\begin{corollary}\label{c:dAb} Let $\G$ be a finitely-generated group. If $\G$ and its abelianisation
have the same rank, then $\widehat{\Gamma}\cong\widehat F_n$ if and only if $\G\cong F_n$.
\end{corollary}

\begin{proof} Every finitely-generated abelian group $A$ has a finite quotient of rank $d(A)$.
\end{proof}

Let $K$ be a finite simplicial graph
with vertex set $V=\{v_1,\dots,v_n\}$ and edge set $E\subset V\times V$.  
Then the {\em right angled Artin group} (or {\em RAAG\/})  associated with $K$
is the group $A(K)$ given by the following presentation:
$$
A(K)=
\< \,v_1,\dots,v_n \mid\, [v_i,v_j]=1 \hbox{  for all  } i, j \hbox{ such that }  \{v_i,v_j\}\in E\,\>.
$$
For example, if $K$ is a graph with $n$ vertices and no
edges, then $A(K)$ is the free group of rank $n$, while if $K$ is the complete
graph on $n$ vertices, then $A(K)$ is the free abelian group $\Z^n$ of rank $n$. 

\smallskip

If the group $\G$ has a presentation of the form $\<A \, | \, R\>$ where $A$ is finite and all of the relators
$r\in R$ lie in the commutator subgroup of the free group $F(A)$, then both $\G$ and its 
abelianisation (which is free abelian) have rank $|A|$. 
The standard presentations of RAAGs and closed surface groups are of this form. 
Also the fundamental group of a closed non-orientable surface has a presentation of the
form $\G=\<\,a_1,b_1,\dots,a_g,b_g,c\mid c^2 = \prod_i[a_i,b_i]\,\>$, so again $\G$
has the same rank as its abelianisation.

\begin{proposition}\label{p:artin} If $\G$ is a right-angled Artin group
that is not free, or 
the fundamental group of a closed surface, then there exists no 
free group $F$ such that $\widehat F\cong \widehat{\Gamma}$.
\end{proposition}

\begin{remark} 
Proposition~\ref{p:epic}
 shows that if $\G$ is a finitely-generated, residually finite
group that is not free, but $\widehat{\Gamma}$ is a free profinite group, then $d(\G) > \widehat d(\widehat{\Gamma})$.
In this context, it is worth noting that for an arbitrarily large integer $d$, 
there exist residually finite hyperbolic groups $\G$ such that $d(\G)=d$ but $\widehat d(\widehat{\Gamma})= 2$.
To construct  such groups, one can follow the first steps of the main
construction in \cite{BG}.
For each integer $d\ge 3$ there exists a
 finitely-presented group $Q$ with $d(Q)=d$ such that
$\widehat Q=\{1\}$ and $H_2(Q,\Z)=0$. The
 Rips construction provides  a short exact
sequence $1\to N\to \G\to Q\to 1$ such that  $\G$ satisfies
a strict small-cancellation condition
and $N$ is a 2-generator group. Wise \cite{wiseSC} proves that the hyperbolic
group $\G$ is the fundamental group of a compact non-positively curved cube complex,
and Agol \cite{agol} proves that such groups are linear, hence residually-finite.
It is shown in \cite{BG} that   $\widehat N\cong\widehat{\Gamma}$, 
from which it follows that $\widehat d(\widehat \G)=2$. 
But $\G$ maps onto $Q$, and therefore $d(\G)\ge d$.
\end{remark}

\subsection{The Parity of Virtual Betti Numbers}

\begin{lemma}\label{even} Let $F$ be a free group of rank $r\ge 1$. If $\widehat{\Gamma}\cong\widehat F$, then
for every positive integer $d$ there exists a subgroup $\Lambda$ of index $d$ in $\Gamma$
such that $b_1(\Lambda) = d(r-1)+1$. In particular, $\Gamma$ has
subgroups of finite index with odd first betti number.
\end{lemma} 

\begin{proof} 
First, $F$ has subgroups of every possible index, 
and every subgroup of index $d$ is free of rank $d(r-1)+1$. 
Next, by Corollary~\ref{same}, for every subgroup $H < F$ of index $d$ there is a 
subgroup $\Lambda <\Gamma$ of index $d$ such that $\widehat{H} \cong \widehat{\Lambda}$, 
and then by Lemma~\ref{sameb1} we have $b_1(\Lambda)=b_1(H) = d(r-1)+1$. 
\end{proof}

The first betti number of the fundamental group of  a closed orientable surface is even, and 
any subgroup of finite index in such a group is again the fundamental group of  a closed orientable surface.
Thus Lemma~\ref{even} provides us with a second proof that surface groups do not have the same
profinite completions as free groups. This argument can be extended to the fundamental groups of
compact {\em K\"ahler manifolds}. Recall that a
smooth manifold $X$ with a Riemannian metric $g$,
symplectic form $\omega$, and complex structure $J$ is said to be K\"ahler if $g(J(u),J(v))=g(u,v)$ and
$g(J_p(u),v)=\omega(u,v)$ for all $p\in X$ and all tangent vectors $u,v\in T_p(X)$. Although every finitely-presented
group occurs as the fundamental group of a closed symplectic manifold and a closed complex manifold,
the K\"ahler condition imposes significant constraints (see \cite{burger}). For example:

\begin{proposition}\label{p:kahler}
If $\G$ is the fundamental group of a compact K\"ahler manifold, then 
there exists no non-trivial free group $F$ such that $\widehat F\cong\widehat{\Gamma}$.
\end{proposition}

\begin{proof} Hodge theory implies that the torsion-free rank of the
  first homology of a compact K\"ahler manifold is even (see
  \cite{burger}, p.7), and this is the first betti number of its
  fundamental group $\G$. A finite-sheeted covering of a compact
  K\"ahler manifold is again a compact K\"ahler manifold, so the first
  betti number of any subgroup of finite index in $\G$ is also
  even.  By the preceding lemma, this completes the proof.
\end{proof}

\subsection{Goodness and surface subgroups}

Goodness was defined in Section~\ref{s:good}. By taking $\G=S$ in the
following proposition, we obtain a further proof that a surface group
cannot have the same profinite completion as a free group.  We shall
then extend this result to (conjecturally all) residually free groups.

A group $\G$ is termed a {\em Poincar\'e duality group\/} over a field ${\bf F}$
(or ${\rm{PD}}_d({\bf F})$ for short) if
there is an integer $d$ such that $H^*(\G; {\bf F}) = H_{*-d}(\G; {\bf F})$. The prototypes
for such groups are the fundamental groups of closed aspherical manifolds.
We will write ${\bf F}_p$ for the field with $p$ elements.

\begin{proposition}\label{p:goodS} 
Let  $\G$ be a group that contains a subgroup $S$ that is ${\rm{PD}}_d({\bf F}_p)$
 for some $d>1$ and some prime $p$. 
 If $S$ is good and  $\widehat{S}\to \widehat{\Gamma}$ is injective, then 
 there does not exist a free group $L$ such that
$\widehat{\Gamma}\cong\widehat L$.
\end{proposition}

\begin{proof} Since free groups are good,
$H^q(\widehat L;{\bf F}_p)=0$ for all $q\ge 2$.
Hence if $\widehat{\Gamma}$ were isomorphic to $\widehat L$, then from Corollary~\ref{cdboundapply}
we would have $H^2(H;{\bf F}_p)=0$ for all $q\ge 2$ and all good
subgroups $H<\G$
 with $\widehat{H}\to \widehat{\Gamma}$ injective.
Now $S$ is such a group, but $H^d(S;{\bf F}_p)= H_0(S; {\bf F}_p)={\bf F}_p$.
\end{proof}

As noted earlier, it was established in \cite{GJZ} that surface groups are good. 
The following lemma includes the case where $\G$ itself is a surface group.

\begin{corollary} If $\G$ contains a surface group $S$, and 
$\widehat{S}\to \widehat{\Gamma}$ is injective,
then $\widehat{\Gamma}$ is not isomorphic to the profinite completion of any free group.
\end{corollary}
 
A similar argument shows the following:

\begin{lemma} If $L$ is a non-abelian free group, then  
$\widehat L$ does not contain  the profinite completion of any  
 surface group, nor that of any free abelian group of rank greater than $1$.
\end{lemma}

\begin{remark} Note that $\widehat L$ does contain
surface subgroups $S$ of arbitrary large genus (as shown in \cite{BGSS} for example)
and free abelian subgroups of arbitrary rank, but
the natural map $\widehat{S}\to \widehat L$ is never injective. The 
surface subgroup examples of \cite{BGSS} are in fact dense in $\widehat{L}$.
\end{remark} 

\subsection{Word hyperbolic groups} Next we single out a particular
case of an application of Proposition~\ref{p:goodS} that is worth
recording. This result connects to
two well-known open problems about word hyperbolic groups, namely:

\begin{enumerate}
\item[\ (A)]  Does every 1-ended word-hyperbolic group contain a surface 
subgroup?
\item[\ (B)]  Is every word-hyperbolic group residually finite?
\end{enumerate}

The first question, due to Gromov, was motivated by the case of hyperbolic 3-manifolds,
and in this special case the question was settled  recently by Kahn and Markovic \cite{KaMa}. 
Indeed, given \cite{KaMa}, a natural strengthening of (A) above is 
to ask: 
\begin{enumerate}
\item[\ \ (A$^\prime$)]  Does every 1-ended word-hyperbolic group contain a 
quasi-convex surface subgroup?
\end{enumerate}

\begin{theorem}
\label{wordhyp}
Suppose that every 1-ended word-hyperbolic group is residually finite and contains a 
quasi-convex surface subgroup. 
Then there exist no 1-ended 
word-hyperbolic group $\G$ and free group $F$ such that $\wh{\G}\cong\wh{F}$. 
\end{theorem}

\begin{proof} Assume the contrary, and let $\G$ be a counter-example, 
with $\widehat{\Gamma} \cong \widehat F$ for some free group $F$.
Let $H$ be a quasi-convex surface subgroup of $\G$. 
Note that the finite-index subgroups of $H$ are also quasi-convex in $\G$. 
Under the assumption that all 1-ended hyperbolic groups are residually finite, it
 is proved  in \cite{AGM} that $H$ and all 
its subgroups of finite index must be separable in $\G$. 
Hence by Lemma~\ref{inducefulltop}, 
the natural map  $\widehat{H}\rightarrow \overline{H} <\widehat{\G}$ is an
isomorphism.  But now Proposition~\ref{p:goodS} yields a contradiction.
\end{proof}

\begin{corollary}
\label{candidate_counterex}
Suppose that there exists a 1-ended word hyperbolic group $\G$ with 
$\widehat{\G}\cong\widehat{F}$ for some free group $F$. Then either
there exists a word-hyperbolic group that is not residually finite, or
there exists a word-hyperbolic group that does not contain
a quasi-convex surface subgroup.
\end{corollary}

\subsection{Residually free groups}

A group $\G$ is called {\em residually free} if for every non-trivial element
$g\in\Gamma$ there is a homomorphism $\phi_g$ from $\G$ to a free
group such that $\phi_g(g)\neq 1$, and $\G$ is {\em fully residually
  free} if for every finite subset $X\subseteq\Gamma$ there is a
homomorphism from $\G$ to a free group that restricts to an injection
on~$X$.  

Following Sela \cite{sela}, we use the term {\em limit group}
to describe a finitely-generated group $\G$ that is fully residually
free.  This defines a rich class of groups, the prototypes for which are
the free abelian groups and torsion-free Fuchsian groups (with the
exception of cocompact groups of Euler characteristic $-1$).
Indeed an arbitrary limit group can be built from these basic examples
in a simple hierarchical manner, by starting with a join of $n$-tori, closed surfaces
and compact graphs, and at each level attaching a torus along a coordinate circle
to the space at the previous level, or attaching a hyperbolic
surface along a boundary curve; one requires that the new
space retracts onto the previous level. The finitely-generated
subgroups of the fundamental groups of these spaces are precisely the
limit groups; see \cite{sela,KM}.

\begin{lemma}\label{l:r-free} If $\G$ is finitely-generated and residually free, then either $\G$
is a hyperbolic limit group, or else $\G$ has a free abelian subgroup of rank $2$.
\end{lemma} 

\begin{proof} Every finitely-generated residually free group $\G$ is a
  subgroup of a direct product of finitely many limit groups
  \cite{brm}. By projecting away from factors that it intersects
  trivially, we may assume that $\G$ intersects each of the
  factors. Limit groups are torsion-free, so if $\G$ intersects two
  factors then it contains a copy of $\Z^2$. Otherwise, it is a
  subgroup of a limit group. Finitely-generated subgroups of limit
  groups are themselves limit groups, and non-hyperbolic limit groups
  contain $\Z^2$.
\end{proof}

Wilton \cite{Wil} proved that if $\G$ is a limit group and $H$ is a 
finitely-generated subgroup of $\G$, then there is a subgroup $\G_0<\G$ of finite index 
that contains $H$ and retracts onto it. If one assumes only that $\G$
is finitely-generated and residually free, then the analogous
statement is false in general. However, it  is shown in \cite{BW} 
that, in this generality, $\G$ does virtually retract onto any
subgroup that has a finite classifying space. Thus we obtain the following:

\begin{theorem}\label{wilton} If $\G$ is finitely-generated and
  residually free, then either
$\G$ has a subgroup of finite index that retracts onto a surface group (of genus $g\ge 1$)
or else $\G$ is a hyperbolic limit group that does not contain a surface subgroup.
\end{theorem}

\begin{theorem} Let $\G$ be a residually-free group and let $F$ be a free group.
If $\widehat{\Gamma}\cong\widehat F$, then either $\G\cong F$ or else $\G$ is a hyperbolic limit
group that does not contain a surface subgroup.
\end{theorem}

\begin{proof} It is easy to see that if $S<\G$ is a virtual retract
  then the natural map $\widehat{S}\to \overline S <\widehat \G$ is an
  isomorphism onto its image, so in the light of Theorem~\ref{wilton}
  we can apply Proposition~\ref{p:goodS}.
\end{proof}

\begin{remark}
It is believed by many that the second possibility stated in the above theorem
does not arise, but this remains unknown for the moment.
Using the hierarchical structure of limit groups described above, one sees
that the crucial case to consider is one where $\G$ is the fundamental group of
a graph of free groups with cyclic edge groups.
There has been much recent work on the existence of surface
subgroups in such groups; see \cite{Cal}, \cite{GW}, \cite{KW}, and \cite{KO}
for example.  
\end{remark}

\begin{remark}\label{r:PvFxF}
  In the preceding proof, we exploited the fact that free groups and
  surface groups are good.  The direct product of two good groups is
  again good (see \cite{GJZ} Proposition 3.4), 
so $F_k\times F_k$ is good. Now consider the fibre product
  $P=\{(x,y)\mid p(x)=p(y)\} <F_k\times F_k$ associated with an
  epimorphism $p:F_k\to Q$, where $Q$ is a finitely-presented group
  with $H_2(Q,\Z)=0$ that has no non-trivial finite quotients.  The
  inclusion $u:P\to \G:=F_k\times F_k$ induces an isomorphism
  $\widehat u: \widehat P \to \widehat \G$ (see \cite{BG}).  Taking
  coefficients in a finite module $M$ with trivial action,  we find that
  $H^2(\widehat{\Gamma}, M)$ will be finitely-generated, since $\G$ is
  good. But a simple spectral sequence calculation shows that in many
  cases $H^2(P,M)$ is {\em not} finitely-generated
  (cf.~\cite{BaumslagRoseblade}), and so $P$ is {\em not} good. Thus we
  have an example of a finitely-generated residually free group that
  is {\em not} good. In contrast, all limit groups {\em are} good; 
  see \cite{GJZ}.

It follows that goodness is definitely a property of $\G$ and not of $\widehat{\Gamma}$:  we
have $P<\G$ with $\widehat P\cong\widehat{\Gamma}$, but $\G$ is good while $P$ is not.
\end{remark}

\section{Torsion in the profinite completions of Fuchsian groups}

In this section we will consider the case of a non-elementary Fuchsian group $\Gamma$.
In the cocompact case, any such group $\Gamma$ has a 
presentation of the form:
$$
\langle a_1, b_1,\ldots, a_g, b_g,x_1,\ldots, x_r~|~x_i^{m_i} = 1~
(i = 1, \ldots, r),~ \prod_1^rx_i\prod_1^g[a_k,b_k] = 1\rangle.
$$
The parameters appearing in the presentation $(g; m_1,\ldots, m_r)$
constitute the {\it signature} of  $\Gamma$.  The integers
$m_i$ are the {\em periods} of $\G$, and $g$ is its
{\em genus}. Groups with signature $(0;p,q,r)$ are generally referred to as {\em
  $(p,q,r)$-triangle groups}. 
If the Fuchsian group is not cocompact, then it is a free
product of cyclic groups. 

\smallskip
We write $\conj(\G)$ to denote the set of conjugacy 
classes of {\em maximal finite subgroups} of a group $\G$. 
Recall that a group $\G$ is called {\em conjugacy separable} if it has the property that 
if the images of two elements of $\G$ are conjugate in every finite quotient 
of $G$ (equivalently, are conjugate in $\widehat{\G}$), then those two elements  
are conjugate in $\G$. 

\begin{theorem}\label{p:conjClasses} 
If $\G$ is a finitely-generated Fuchsian group, then the natural
inclusion $\G\to\wh\G$ induces a bijection 
$\conj(\G)\to\conj(\wh\G)$. More precisely,
every finite subgroup of $\widehat{\Gamma}$ is conjugate to
a subgroup of $\G$, and if two maximal finite subgroups of $\G$ are conjugate in $\widehat{\Gamma}$ then they are
already conjugate in $\G$.
\end{theorem}

\begin{proof} 
  The finite subgroups of Fuchsian groups are cyclic, so the
  injectivity of $\conj(\G)\to\conj(\wh\G)$ is a special case of the
  fact that finitely-generated Fuchsian groups are conjugacy separable \cite{St}. 

\smallskip
In order to prove surjectivity, we first suppose that
  $\G$ is not a triangle group. In this case, $\G$ splits as an
  amalgamated free product $\G=A\ast_C B$, where the non-trivial groups
  $A$ and $B$ are free products of cyclic groups.  By Bass-Serre
  theory, the finite free factors of $A$ and $B$ are in bijection with
  the conjugacy classes of maximal finite subgroups in $\Gamma$.  Now by 
  Corollary~\ref{profinitesurface}, we have $\widehat{\Gamma} \cong
  \widehat{A} \amalg_{\widehat{C}}\widehat{B}$ and by Theorem
 ~\ref{finite_in_vertex}, every finite subgroup of $\widehat{\Gamma}$
  is conjugate to a subgroup of $\widehat A$ or $\widehat B$.
  Repeating this argument with $A$ and $B$ in place of $\Gamma$, we
  deduce that every finite subgroup of $\widehat \Gamma$ is conjugate
  to a subgroup of the profinite completion of one of the free factors
  of $A$ or $B$. But each free factor $Z_i$ is cyclic, so $\widehat
  Z_i$ is torsion-free ($\cong\widehat\Z$) if $Z_i$ is infinite, while 
  $\widehat Z_i =Z_i$ if $Z_i$ is finite. Thus every maximal finite
  subgroup of $\widehat{\Gamma}$ is conjugate to one of the $Z_i$.

It remains to prove surjectivity in the case where $\G$ is a hyperbolic 
triangle group, say the $(p,q,r)$-triangle group. 
Let $H$ be a maximal finite subgroup of $\widehat \G$.
Since $\widehat{\Gamma}$ is residually finite, we can pass to a subgroup of arbitrarily large index that contains $H$; 
in particular, we can make the index greater than $pqr$. Let $\Lambda\subset\G$ be
the intersection of $\G$ with this subgroup. Then the index of $\Lambda$ 
 in $\Gamma$ is greater than $pqr$, so the rational Euler characteristic of $\Lambda$
 is less than $-2$, which means that $\Lambda$ is not a hyperbolic triangle group. 
Hence our previous argument applies, and $H$ is conjugate in $\widehat{\Lambda}$ 
to a finite subgroup $H'$ of $\Lambda$. Since $H$ is maximal in $\widehat{\Gamma}$,
it follows that this $H'$ is maximal in $\G$.
\end{proof}

We can now use Theorem \ref{p:conjClasses} to eliminate the latter
case of Theorem \ref{patch}.

\begin{corollary}
\label{eliminatefinitenormal}
Let $\G$ be a finitely generated Fuchsian group. Then $\wh\G$ does not
contain a non-trivial finite normal subgroup.\end{corollary}

\begin{proof}
Suppose $N$ is a non-trivial finite normal subgroup of $\wh\G$. By Theorem
\ref{p:conjClasses}, since every finite subgroup of $\wh\G$ is conjugate to
a subgroup of $\G$, $N$ is a subgroup of $\G$.  However,
non-elementary Fuchsian groups do not contain finite, non-trivial normal
subgroups.\end{proof}

Also, note that Theorem  \ref{p:conjClasses}
provides another proof of the following special case of Corollary~\ref{notorsion2}.

\begin{corollary}\label{c:hatTF} 
If $\G$ is a torsion-free Fuchsian group, then $\widehat{\Gamma}$ is torsion-free.
\end{corollary}

\section{The Proof of Theorem~\ref{main1}}

In this section we complete the proof of Theorem~\ref{main1}, re-stated below for
convenience.
\medskip

\noindent{\bf{Theorem.}} {\em Let $\Gamma_1$ be a finitely-generated Fuchsian group and
let $\Gamma_2$ be a lattice in a connected Lie group.  
If ${\cal C}(\Gamma_1) = {\cal C}(\Gamma_2)$, then 
$\Gamma_1 \cong \Gamma_2$.}
\medskip
%

\begin{proof}[{\bf Proof of Theorem~\ref{main1}}] 
By Theorem \ref{patch} and Corollary \ref{eliminatefinitenormal}, we can 
now assume that both $\G_1$ and $\G_2$ are Fuchsian.

\smallskip 

Now suppose that $\G_1$ and $\G_2$ are torsion-free.  
In this case, each of $\Gamma_1$ and
$\Gamma_2$ is either a free group or a surface group --- that is, isomorphic
to the fundamental group of a closed orientable surface of genus at
least $2$. We proved (three times) in Section~\ref{s:FvS} that if
$\widehat{\Gamma}_1\cong\widehat \G_2$ then both of $\G_1$ and $\G_2$ are 
free groups, or both are surface groups.  But the isomorphism type of a free group
or a surface group $\G$ is determined by the rank of its abelianisation, 
which by Corollary~\ref{sameb1} is an invariant of~$\widehat{\Gamma}$.
Hence the theorem holds in this case. 

Next, we note that it cannot be true that $\Gamma_1$ is cocompact and
$\Gamma_2$ is not. For if this were the case, then we could pass to
torsion-free subgroups of common finite index that would still have
isomorphic profinite completions (see Corollary~\ref{same}), which is
impossible by the previous argument.
 
If neither $\G_1$ nor $\G_2$ is cocompact, then each is a free product of 
cyclic groups. We know that $b_1(\G_1)=b_1(\G_2)$, by Corollary~\ref{sameb1},
and so the number of infinite cyclic factors in each product is the same. And by Theorem
\ref{p:conjClasses}, the finite cyclic factors, being in bijection with the conjugacy
classes of maximal finite subgroups, are also the same. 
Hence the theorem holds in this case too. 

It only remains to consider the case where both $\G_1$ and $\G_2$ are cocompact
groups with torsion (by Corollary \ref{c:hatTF}). 
The genus of $\G_i$ is determined by $b_1(\G_i)$,
so  Corollary~\ref{sameb1} tells us that $\G_1$ and $\G_2$ are of the same genus.
The periods of $\G_i$ are the orders of representatives of the conjugacy classes of maximal finite
subgroups of $\G_i$, and so by Proposition~\ref{p:conjClasses} these must also be the
same for $\G_1$ and $\G_2$. Thus $\G_1$ and $\G_2$ have the same signature, and
are therefore isomorphic.

This completes the proof of  Theorem~\ref{main1}. 
\end{proof}

\begin{remark}~It is interesting to compare Theorem~\ref{main1} with
  some of the findings in \cite{GZ}.  Among other things, it is shown in
  \cite[Theorem 4.1]{GZ} that if $\G=A*B$ is a free product of finite
  groups $A$ and $B$, then among all finitely-generated virtually free
  groups, $\G$ is determined by its profinite completion. Our Theorem~\ref{main1}
covers free products of finite cyclic groups with arbitrarily many
factors. On the other hand, examples are provided in \cite{GZ} of finitely-generated virtually
free groups that are not determined by their profinite completions.
\end{remark}

\section{Distinguishing non-uniform arithmetic Kleinian groups}\label{s:psl2C}

A non-uniform arithmetic lattice in $\PSL(2,{\Bbb C})$ is a discrete
group commensurable (up to conjugacy) with a
Bianchi group $\PSL(2,O_d)$, where $O_d$ is the ring of integers in
the quadratic imaginary number field ${\Bbb Q}(\sqrt{-d})$; see \cite{MR}.
In this section we prove Theorem~\ref{mainbianchi}, re-stated below for 
convenience:

\medskip
\noindent{\bf{Theorem.}}
{\em Let $\Gamma_1$ be a non-uniform lattice in $\PSL(2,{\CC})$, 
and let $\Gamma_2$ be a non-uniform irreducible 
arithmetic lattice in a semisimple Lie group $G$ having 
trivial centre and no compact factors. 
If ${\cal C}(\Gamma_1) =
{\cal C}(\Gamma_2)$ then $G\cong \PSL(2,{\CC})$. Moreover, if $\G_1$
is arithmetic then the family of 
all $\G_2$ with ${\cal C}(\Gamma_1) = {\cal C}(\Gamma_2)$ divides into
finitely many commensurability classes.}

\medskip

\noindent 
After a series of reductions, we shall arrive at the case where $\G_2$
is a torsion-free lattice in $\SO_0(n,1)$ or $\SU(m,1)$. 
The {\em peripheral subgroups} of $\G_2$ are the maximal subgroups
consisting entirely of parabolic elements; they are the fundamental groups of the
cusps of the locally symmetric space with fundamental group $\G_2$.
If $\G_2<\SO_0(n,1)$ then  these subgroups are virtually $\Z^{n-1}$,
while if $\G_2<\SU(n,1)$ they are commensurable with the integral Heisenberg group $H_{2m-1}$,
which is a central extension of $\Z^{2m-2}$ by $\Z$. 
We shall make use of the following: 

\begin{lemma}
\label{peripheralseparability} 
Let $\G$ be a non-uniform arithmetic lattice in
$\SO_0(n,1)$ or $\SU(m,1)$ and let $D$
be a peripheral subgroup of $\G$. 
Then every subgroup of $D$ is separable in $\Gamma_2$, 
and $D$ is good.
\end{lemma}

\begin{proof} The first part of the lemma is due to McReynolds; 
see  \cite[Theorem 1.3 and Corollaries 4.1 and 4.2]{McR}.
For the second part, it suffices to show that the groups $\Z^k$ and
$H_k$ are good, because goodness is preserved by commensurability; 
see \cite[Lemma 10]{GJZ}.  As noted above, both of these groups are extensions of
a free abelian group by $\Z$, and so they are good by Lemma~\ref{goodexact}.
\end{proof}
%

We will also use this theorem, which was proved in \cite[Corollary 1.2]{CLR}:

\begin{theorem}
\label{boundthenumberofcusps}
Let $N$ be any positive integer. There are only finitely many commensurability
classes of non-uniform arithmetic lattices in $\PSL(2,{\CC})$ that
contain a lattice $\Gamma$ with the property that ${\bf H}^3/\Gamma$ has at most $N$ cusps.
\end{theorem}

\begin{proof}[{\bf Proof of Theorem~\ref{mainbianchi}}] 
As before, Theorem~\ref{allfinite} implies that $\widehat{\G}_1\cong \widehat{\G}_2$.
First note that if necessary we can replace $\G_1$ and $\G_2$ by
subgroups of finite index so as to assume that they are
torsion-free. Indeed, if we first pass to a finite index subgroup of
$\G_1$ (which we will continue to refer to as $\G_1$)  to arrange
that $\G_1$ is torsion-free, then Corollary~\ref{same} will provides us with
a subgroup of finite index in $\G_2$ (which again we will continue to
refer to as $\G_2$) so that $\widehat{\G}_1\cong \widehat{\G}_2$. 
At this point, if $\G_2$ is not torsion-free, then we can pass to a further subgroup of
finite index in $\G_2$ that is torsion-free, and use Corollary
\ref{same} again to pass to the corresponding subgroup of $\G_1$.

A standard Alexander-Lefschetz duality argument 
implies that $b_1(\Gamma_1)$ is positive,
so $b_1(\Gamma_2)$ is positive too, by Corollary~\ref{sameb1}. 
Given this, it follows that $\G_2$ cannot be a
lattice  in a group of real rank at least $2$, because the Margulis normal
subgroup theorem \cite{Marg} implies that such groups 
have finite abelianisation (note that this is where irreducibility is invoked).
Similarly, $\G_2$ cannot be a lattice in a rank 1 Lie group with property (T), 
and this rules out ${\rm{Sp}}(n,1)$ ($n\geq 2$) and the isometry
group of the Cayley hyperbolic plane. The
only remaining possibilities (up to finite index) are 
$(\P)\SO_0(n,1)$ and $(\P)\SU(n,1)$. 

We dealt with $\SO_0(2,1)=\PSL(2,{\R})$ in Theorem~\ref{main1}.  
Hence it remains to rule out $\SO_0(n,1)$ for $n\geq 4$ and $\SU(m,1)$
for $m\geq 2$. In these cases, as discussed above, we have a peripheral
subgroup $D<\G_2$ that is good, with the property that the closure of 
its image in $\widehat{\Gamma}_2$ is isomorphic to $\widehat D$. 
But $D=\Z^{n-1}$ or $H_{2m-1}$ is the fundamental group of a closed,
orientable manifold, of dimension $n-1\ge 3$ in the former case, and
dimension $2m-1\geq 3$ in the latter. As such, $D$ satisfies Poincar\'e
duality with coefficients in an arbitrary finite field. It therefore
follows from Corollary~\ref{cdboundapply} that $\widehat{\Gamma}_2$
cannot be isomorphic to the profinite completion of a good group of
cohomological dimension less than $3$. On the other hand, $\G_1$ is such a group: it
has cohomological dimension $2$ because it is the fundamental group of
an aspherical 3-manifold with non-empty boundary, and it is good by
Theorem~\ref{good}.  At this point, we have proved that $\G_2$ is a
lattice in ${\rm{SO}}_0(3,1)\cong{\rm{PSL}}(2,\Bbb C)$.
%
%
%

To complete the proof of Theorem~\ref{mainbianchi}, we suppose now that
$\G_1$ and $\G_2$ are non-uniform arithmetic lattices in $\PSL(2,\Bbb C)$
and that $b_1(\Gamma_1)=n$. Then $b_1(\Gamma_2) = n$ by Corollary
\ref{sameb1}, and so by a standard fact about 3-manifold
groups with torus boundary components, the number of cusps of 
${\bf H}^3/\Gamma_2$ is at most $n$. 
Application of Theorem~\ref{boundthenumberofcusps} now  completes the 
proof of Theorem~\ref{mainbianchi}. 
\end{proof}

If we drop the assumption of irreducibility of the lattice $\G_2$ 
in Theorem~\ref{mainbianchi}, then some observations can be made, 
but at this stage we are unable to reach the same conclusion. 
For instance, at present we are unable to answer the following
question.

\begin{question}
\label{product}
Let $\Gamma_1$ be a non-uniform lattice in $\PSL(2,{\CC})$, 
and let $\Gamma_2=F\times \Delta$ where $F$ is a free group of rank $r > 1$ 
and $\Delta$ is a cocompact lattice in $\SU(n,1)$.
Is it possible that $\widehat{\Gamma}_1\cong \widehat{\Gamma}_2$?\end{question}

\section{Finite quotients of triangle groups}

In this section we give a more direct proof  of the fact that triangle groups
are distinguished among themselves by their finite quotients, and give some 
explicit quotients that distinguish non-triangle groups from triangle groups.

\begin{theorem}
\label{trianglegroups}
If $\Gamma$ and $\Sigma$ are triangle groups 
for which ${\cal C}(\Gamma) = {\cal C}(\Sigma)$, then $\Gamma\cong \Sigma$.
\end{theorem}

We will use the notation 
$\Delta(r,s,t) = \langle\, x,y,z \, | \, xyz = x^r = y^s = z^t = 1 \,\rangle$ 
for the $(r,s,t)$ triangle group.
Each triangle group is called {\em spherical}, {\em Euclidean}  
or {\em hyperbolic\/} according to whether the quantity $1/r+1/s+1/t$ is greater 
than, equal to or less than 1, respectively. 
Note that $\Delta(r,s,t)$ is isomorphic to $\Delta(u,v,w)$ whenever the triple 
$(u,v,w)$ is a permutation of the triple $(r,s,t)$, and hence up to isomorphism we may assume 
that $r \le s \le t$. 

The spherical triangle groups are $\Delta(1,n,n)$, $\Delta(2,2,n)$, $\Delta(2,3,3)$, 
$\Delta(2,3,4)$ and $\Delta(2,3,5)$, which are isomorphic to $C_n$ (cyclic), 
$D_n$ (dihedral of order $2n$), $A_4$, $S_4$ and $A_5$, respectively. 
The Euclidean triangle groups are  $\Delta(2,3,6)$, $\Delta(2,4,4)$ and $\Delta(3,3,3)$, 
each which is an extension of a free abelian group of rank $2$ by a cyclic group $C_t$ 
(with $t = 6, 4$ and $3$, respectively). In particular, the spherical triangle groups 
are finite, while the Euclidean triangle groups are infinite but soluble.  In contrast, 
all hyperbolic triangle groups are infinite but insoluble.  See \cite{CM} for further details. 

The latter categorisation makes the spherical and Euclidean triangle groups easy to distinguish 
from others by their finite quotients, and so we will restrict our attention to the hyperbolic 
ones, which are the Fuchsian groups with signature $(0;r,s,t)$ where $1/r+1/s+1/t < 1$.  
The most famous of these is $\Delta(2,3,7)$, as it gives the largest value of $1/r+1/s+1/t$ 
less than $1$, and its non-trivial quotients are the {\em Hurwitz groups} (see \cite{Con}). 

We will define a finite group $G$ to be $(r,s,t)$-{\em generated\/} if $G$ can be generated 
by elements $a$, $b$ and $c$ of (precise) orders $r$, $s$ and $t$ such that $abc = 1$.
For any hyperbolic triple $(r,s,t)$, the set of $(r,s,t)$-generated groups is non-empty, 
by residual finiteness of $\Delta(r,s,t)$, but in most cases $\Delta(r,s,t)$ can 
also have `non-smooth' quotients, in which the orders of the generators are not preserved. 

The following theorem is a direct consequence of observations 
made by Macbeath \cite{Mac} on $(r,s,t)$-generation of the groups $\PSL(2,q)$, 
and will be critical to our proof of Theorem~\ref{trianglegroups}: 

\begin{theorem}
\label{macbeathsthm}
Let $(r,s,t)$ be any hyperbolic triple other than $(2,5,5)$, $(3,4,4)$, $(3,3,5)$, $(3,5,5)$ 
or $(5,5,5)$, and let $p$ be any prime.  If $p^f$ is the smallest power of $p$ for 
which $\PSL(2,p^f)$ contains elements of orders $r$, $s$ and $t$, 
then either $\PSL(2,p^f)$ is $(r,s,t)$-generated, 
or $f$ is even and $\PGL(2,p^{f/2})$ is $(r,s,t)$-generated. 
In particular, $\PSL(2,p)$ itself is $(r,s,t)$-generated whenever it contains 
elements of orders $r$, $s$ and $t$. 
\end{theorem}  

The triples $(2,5,5)$, $(3,4,4)$, $(3,3,5)$, $(3,5,5)$ and $(5,5,5)$, together 
with the spherical triples and the triple $(3,3,3)$, were called {\em exceptional} 
by Macbeath.  Note that the group $A_5 \cong \PSL(2,5)$ is 
$(2,5,5)$-, $(3,3,5)$-, $(3,5,5)$- and $(5,5,5)$-generated, while the group $S_4$ 
is $(3,4,4)$-generated.  

\smallskip
We will also make use of the fact that if the finite group $G$ is $(r,s,t)$-generated, 
then $G$ is a group of conformal automorphisms of a compact Riemann surface $S$ 
of genus $g$, where 
$$
2-2g = |G|\left(\frac{1}{r}+\frac{1}{s}+\frac{1}{t}-1\right)  
$$     
as a consequence of the Riemann-Hurwitz formula.
The kernel $K$ of the corresponding smooth homomorphism from $\Delta(r,s,t)$ 
onto $G$ is the fundamental group of $S,$ and is itself a Fuchsian group, 
with signature $(2g; - )$.  In particular, $K$ is generated by $2g$ elements 
$a_1, b_1,\ldots, a_g, b_g$ subject to a single defining relation 
$[a_1,b_1] \ldots [a_g,b_g] = 1$. 
Now for any positive integer $n$, the subgroup $K'K^{(n)}$ generated 
by the derived subgroup $K'$ and the $n$th powers of all elements of $K$ 
is characteristic in $K$ and hence normal in $\Delta(r,s,t)$, and the quotient 
$\Delta(r,s,t)/K'K^{(n)}$ is then isomorphic to an extension by $G$ of an abelian 
subgroup $K/K'K^{(n)}$ of rank $2g$ and exponent $n$ (and order $n^{2g}$). 
Thus for any such $G$, we can construct an infinite family of smooth 
quotients of $\Delta(r,s,t)$, to help distinguish $\Delta(r,s,t)$ from 
other triangle groups. 

\medskip
We can now proceed to prove Theorem~\ref{trianglegroups}. 
To do that, we will suppose that $\Gamma = \Delta(r,s,t)$ and $\Sigma = \Delta(u,v,w)$, 
where $(r,s,t)$ and $(u,v,w)$ are hyperbolic triples with $r \le s \le t$ and $u \le v \le w$, 
and that ${\cal C}(\Gamma) = {\cal C}(\Sigma)$. 
We will prove in steps that $(r,s,t) = (u,v,w)$. 


\begin{lemma}
\label{easyones} 
${}$\\[+2pt]  
{\rm (a)} $\,\gcd(r,s,t) = \gcd(u,v,w)$, \ and \  
{\rm (b)} ${\ds \,\frac{rst}{\lcm(r,s,t)} = \frac{uvw}{\lcm(u,v,w)}}$, \ and also \\[+2pt]   
{\rm (c)} $\,\lcm(\gcd(r,s),\gcd(r,t),\gcd(s,t)) = \lcm(\gcd(u,v),\gcd(u,w),\gcd(v,w))$. 
\end{lemma}

\begin{proof} 
The abelianisation of $\Gamma = \Delta(r,s,t)$ is $C_d \times C_e$, 
where $d = \gcd(r,s,t)$ and $de = rst/\lcm(r,s,t)$ 
and $e = \lcm(\gcd(r,s),\gcd(r,t),\gcd(s,t))$, and similarly for $\Sigma = \Delta(u,v,w)$. 
Since $\Gamma$ and $\Sigma$ have the same finite abelian quotients, the results follow.
\end{proof}

\begin{lemma}
\label{sameeuler}
$\Gamma$ and $\Sigma$ have the same set of 
$(r,s,t)$-generated quotients, and the same set of $(u,v,w)$-generated quotients. 
Moreover, $\frac{1}{r}+\frac{1}{s}+\frac{1}{t} = \frac{1}{u}+\frac{1}{v}+\frac{1}{w}$.
\end{lemma}

\begin{proof}~The first two assertions are easy, since 
$\Gamma$ and $\Sigma$ have exactly the same finite quotients. 
For the last part, without loss of generality we may 
suppose $\frac{1}{r}+\frac{1}{s}+\frac{1}{t} \le \frac{1}{u}+\frac{1}{v}+\frac{1}{w}$. 
Now let $G$ be any $(r,s,t)$-generated quotient of $\Gamma$. 
Then $G$ is also a finite quotient of $\Sigma$.  
Next let $u', v'$ and $w'$ be divisors of $u,v$ and $w$ (respectively) such that 
$G$ is $(u',v',w')$-generated, and such that $\frac{1}{u'}+\frac{1}{v'}+\frac{1}{w'}$ is 
as small as possible subject to those conditions. 
Then in particular, 
$\frac{1}{u'}+\frac{1}{v'}+\frac{1}{w'} \ge \frac{1}{u}+\frac{1}{v}+\frac{1}{w} 
\ge \frac{1}{r}+\frac{1}{s}+\frac{1}{t}$. 

\smallskip
For any $n$ coprime to $|G|$, the largest quotient of $\Gamma$ that is an 
extension of an abelian group of exponent $n$ by $G$ has order $n^{2g}|G|$, 
where $2-2g = |G|(\frac{1}{r}+\frac{1}{s}+\frac{1}{t}-1)$. 
On the other hand, the largest quotient of $\Sigma$ that is an 
extension of an abelian group of exponent $n$ by $G$ must be a smooth quotient 
of the $(u',v',w')$ triangle group and hence has order $n^{2g'}|G|$, 
where $2-2g'=|G|(\frac{1}{u'}+\frac{1}{v'}+\frac{1}{w'}-1)$.

\smallskip
Since $\Gamma$ and $\Sigma$ have the same quotients, 
it follows that $g' = g$, and so 
$\frac{1}{u'}+\frac{1}{v'}+\frac{1}{w'} = \frac{1}{r}+\frac{1}{s}+\frac{1}{t}$. 
The inequality obtained at the end of the first paragraph now gives 
$\frac{1}{u}+\frac{1}{v}+\frac{1}{w} = \frac{1}{r}+\frac{1}{s}+\frac{1}{t}$. 
It also gives $\frac{1}{u'}+\frac{1}{v'}+\frac{1}{w'} = \frac{1}{u}+\frac{1}{v}+\frac{1}{w}$, 
so $(u',v',w') = (u,v,w)$, and hence $G$ is $(u,v,w)$-generated.  

\smallskip
In particular, every $(r,s,t)$-generated finite group is also $(u,v,w)$-generated.
The converse holds by the same argument, with the roles of 
$(r,s,t)$ and $(u,v,w)$ reversed. 
\end{proof}

\begin{lemma}
\label{moreconditions}
$rst = uvw$, and $\lcm(r,s,t) = \lcm(u,v,w)$, 
and $rs+rt+st = uv+uw+vw$.
\end{lemma}

\begin{proof}
For any prime divisor $p$ of $rst$, 
let $p^\alpha$, $p^\beta$ and $p^\gamma$ be the largest powers of $p$ 
dividing $r,s$ and $t$, ordered in such a way that $\alpha \le \beta \le \gamma$.  
Then $p^\alpha$ must be the largest power of $p$ dividing $\gcd(r,s,t)$, 
while $p^\beta$ is the largest power of $p$ dividing $\lcm(\gcd(r,s),\gcd(r,t),\gcd(s,t))$, 
and $p^\gamma$ is the largest power of $p$ dividing $\lcm(r,s,t)$. 
Also $p^{\alpha+\beta+\gamma}$ is the largest power of $p$ dividing $rst$, 
and so $p^{\alpha+\beta}$ is the largest power of $p$ dividing $\frac{rst}{\lcm(r,s,t)}$. 
Furthermore, either $\beta = \gamma$, or $p^\gamma$ is the largest power 
of $p$ dividing the denominator of $\frac{1}{r}+\frac{1}{s}+\frac{1}{t} = \frac{rs+rt+st}{rst}$ 
when the latter is expressed in reduced form.
(To see the last part, note that $rs+rt+st$ is divisible by $p^{\alpha+\beta}$ 
but not $p^{\alpha+\beta+1}$  when $\beta < \gamma$.) 

Hence the largest powers of $p$ dividing $r$,$s$ and $t$ are determined by the 
quantities $\gcd(r,s,t)$, $\frac{rst}{\lcm(r,s,t)}$ and $\frac{1}{r}+\frac{1}{s}+\frac{1}{t}$.  
By Lemmas~\ref{easyones} and~\ref{sameeuler}, these three quantities 
are the same for the triple $(u,v,w)$, and hence the largest powers of $p$ dividing $u$, $v$ 
and $w$ are equal to those for $r$,$s$ and $t$, in some order. 
As this holds for every prime $p$, the stated equalities follow easily. 
\end{proof}

We now deal with many special  cases of Theorem~\ref{trianglegroups}. 

\begin{proposition}
\label{somecases1}
The conclusion of Theorem {\em~\ref{trianglegroups}} holds whenever 

\smallskip
\noindent {\rm (a)}~$(r,s,t)$ is one of the exceptional triples 
$(2,5,5)$, $(3,4,4)$, $(3,3,5)$, $(3,5,5)$ or $(5,5,5)$, or 

\smallskip
\noindent {\rm (b)}~the triples $(r,s,t)$ and $(u,v,w)$ have 
an entry in common, or 

\smallskip
\noindent {\rm (c)}~two or more of $r$, $s$ and $t$ are even.
\end{proposition}

\begin{proof}
%
%
(a)~This follows from Lemma~\ref{moreconditions}, since 
$(2,5,5)$ is the only hyperbolic triple with $rst = 50$, 
and $(3,4,4)$ is the only hyperbolic triple with $rst = 48$, $\lcm(r,s,t) = 12$ and $\gcd(r,s,t) = 1$, 
and $(3,3,5)$ is the only hyperbolic triple with $rst = 45$, 
and $(3,5,5)$ is the only hyperbolic triple with $rst = 75$, 
and finally, $(5,5,5)$ is the only hyperbolic triple with $rst = 125$. \\[+6pt]
\noindent 
(b)~Suppose for example that $t = w$. 
Then $rs = \frac{rst}{t} = \frac{uvw}{w} = uv$, and then since 
$rs+(r+s)t = rs+rt+st = rst(\frac{1}{r}+\frac{1}{s}+\frac{1}{t}) 
= uvw(\frac{1}{u}+\frac{1}{v}+\frac{1}{w}) = uv+uw+vw = uv+(u+v)w$, 
we find that $r+s = u+v$.  
It is now an easy exercise to deduce from $rs = uv$ and $r+s = u+v$ that 
$\{r,s\} = \{u,v\}$, and hence $(r,s,t) = (u,v,w)$. 
The same argument works for all other possible coincidences between 
entries of $(r,s,t)$ and $(u,v,w)$. 
\\[+6pt]
\noindent 
(c)~Suppose two or more of $r$, $s$ and $t$ are even.  
Then $\lcm(\gcd(r,s),\gcd(r,t),\gcd(s,t))$ is even, 
and therefore so is $\lcm(\gcd(u,v),\gcd(u,w),\gcd(v,w))$, 
and hence two or more of $u$, $v$ and $w$ are even.  
Also if all three of $r$, $s$ and $t$ are even, then $\gcd(r,s,t)$ is even, 
hence so is $\gcd(u,v,w)$, and therefore all three of $u$, $v$ and $w$ are even.  
Now let $m = \max(t,w)$ if all three of $r$, $s$ and $t$ are even, 
or otherwise let $m$ be the largest odd integer among $r,s,t,u,v$ and $w$. 
Then the dihedral group $D_m$ of order $2m$, which is $(2,2,m)$-generated, 
is a quotient of $\Gamma$ or $\Sigma$, and hence must also be 
a quotient of the other.  By definition of $m$, it follows that $m$ appears in both 
triples $(r,s,t)$ and $(u,v,w)$, and hence by part (b), we have $(r,s,t) = (u,v,w)$.
\end{proof}
%
%

To continue with the proof, we require some information about
the groups $\PSL(2,p)$, for $p$ prime. 
When $p$ is odd, the orders of the elements 
of $\PSL(2,p)$ are precisely the divisors of $p$, $\frac{p-1}{2}$ 
and $\frac{p+1}{2}$ (see \cite[Chapter 3.6]{Su} for example).   
Note that the integers $p$, $\frac{p-1}{2}$ 
and $\frac{p+1}{2}$ are pairwise coprime, 
so the order of any non-trivial element of $\PSL(2,p)$ divides exactly one 
of them. 

\smallskip
Now define the {\em $L_2$-set\/} of a triple $(k,l,m)$ to be the
(unique) set of pairwise coprime positive integers whose least common
multiple is the same as that of $\{k,l,m\}$ and which has the property
that each of $k,l$ and $m$ divides exactly one member of that set.
For example, if $k,l$ and $m$ are themselves pairwise coprime, then the
$L_2$-set of the triple $(k,l,m)$ is just $\{k,l,m\}$, while if
$\gcd(k,lm) = 1$ but $\gcd(l,m) > 1$ then its $L_2$-set is
$\{k,\lcm(l,m))\}$, and if $\gcd(k,l) > 1$ and $\gcd(l,m) > 1$ then
its $L_2$-set is $\{\lcm(k,l,m))\}$.  Note that every maximal
prime-power divisor of $\lcm(k,l,m)$ divides exactly one member of the
$L_2$-set.  


It follows from Macbeath's theorem (Theorem~\ref{macbeathsthm}) that if the triple $(k,l,m)$ is
non-exceptional, then the group $\PSL(2,p)$ is $(k,l,m)$-generated if
and only if each member of the $L_2$-set of the triple $(k,l,m)$ is
equal to $p$ or a divisor of $\frac{p \pm 1}{2}$. 
This enables us to prove the following: 

\begin{lemma}
\label{sameL2set}
If the triples $(r,s,t)$ and $(u,v,w)$ are non-exceptional, 
then they have the same $L_2$-set. \end{lemma}
 
\begin{proof} Suppose that the $L_2$-sets of $(r,s,t)$ and
$(u,v,w)$ are distinct.  Then, by swapping $(r,s,t)$ and $(u,v,w)$ if
necessary, we see that there must exist maximal prime-power divisors
$q_1$ and $q_2$ of $\lcm(r,s,t) = \lcm(u,v,w)$ such that $q_{1}q_{2}$
divides one member of the $L_2$-set of $(u,v,w)$, but $q_1$ and $q_2$
divide different members of the $L_2$-set of $(r,s,t)$, say $m_1$ and
$m_2$.  Now by the Chinese Remainder Theorem and Dirichlet's theorem
on primes in arithmetic progression, there are infinitely many (odd)
primes $p$ such that $p \equiv 1$ mod $2m_1$ while $p \equiv -1$ mod
$2m_2$ and $p \equiv -1$ mod $2n$ for any other member $n$ of the
$L_2$-set of $(r,s,t)$.  For any such prime $p$, the group $\PSL(2,p)$
is $(r,s,t)$-generated, since $p \equiv \pm 1$ mod $2m$ for every
member $m$ of the $L_2$-set of $(r,s,t)$.  On the other hand, 
$\PSL(2,p)$ has no element of order $q_{1}q_{2}$, since $q_1$ divides
$\frac{p-1}{2}$ while $q_2$ divides $\frac{p+1}{2}$, and therefore
$\PSL(2,p)$ cannot be $(u,v,w)$-generated, contradiction. \end{proof}

\begin{corollary}
\label{coprime}
If one of $r,s,t$ is coprime to each of the other two, then $(r,s,t) = (u,v,w)$.
In particular, if $r,s$ and $t$ are pairwise coprime, 
or equivalently, if the $(r,s,t)$ triangle group $\Gamma$ is perfect, 
then $(r,s,t) = (u,v,w)$.
\end{corollary}

\begin{proof}~Clearly we need only prove the first assertion, 
and then the rest follows.  So suppose that  $(r,s,t) \ne (u,v,w)$, and also, 
say, that $\gcd(r,st) = 1$. (The other two cases are similar.) 
Now $uvw = rst$ by Lemma~\ref{moreconditions}, and each of $u,v$ and $w$ 
is distinct from $rs$ and $t$, by Proposition~\ref{somecases1}(b). 
Hence at least one of $u,v$ and $w$ divides neither $r$ nor $st$, 
and so must be of the form $cd$, where $c$ and $d$ 
are non-trivial divisors of $r$ and $st$ respectively. 
It follows that the $L_2$-sets of $(r,s,t)$ and $(u,v,w)$ are distinct, 
contradiction. \end{proof}

The next step requires a further general observation about $L_2$-sets.

\begin{lemma} 
\label{nondivisor} 
For every triple $(k,l,m)$ such that $k,l,m > 1$ and at most one of $k,l,m$ is even, 
and for every integer $q > 3$ that does not divide any of the members of 
the $L_2$-set of $(k,l,m)$, there exists a finite quotient $G$ of the $(k,l,m)$ 
triangle group such that $G$ has no element of order $q$.  
\end{lemma}

\begin{proof} 
This is easy to see for the exceptional triples: 
we can take $G = A_4$ for $(k,l,m) = (2,3,3)$,  
or $G = S_4$ for $(k,l,m) = (2,3,4)$, 
or $G = A_5$ for $(k,l,m) = (2,3,5)$, $(2,5,5)$, $(3,3,5)$, $(3,5,5)$ or $(5,5,5)$  
(and also $G = C_3 \times C_3$ for $(k,l,m) = (3,3,3)$).
For any non-exceptional triple $(k,l,m)$, we can take $G = \PSL(2,p)$, where $p$ is 
a prime such that $p \equiv \pm 1$ modulo twice each of  the members of 
the $L_2$-set of $(k,l,m)$, but $p \not\equiv \pm 1$ modulo $2q$. 
\end{proof}

\begin{corollary}
\label{nodivisors} 
The integers $u,v$ and $w$ do not have non-trivial divisors $u',v'$ and $w'$ 
such that one of $r,s$ and $t$ is coprime to each of $6,u',v'$ and $w'$. 
\end{corollary}

\begin{proof} 
Suppose to the contrary that $q \in \{r,s,t\}$ is coprime to $6,u',v'$ and $w'$. 
Then there exists a finite quotient $G$ of the triangle group $\Delta(u',v',w')$ such 
that $G$ has no non-trivial element of order dividing $q$. But then this group $G$ 
is a quotient of $\Sigma = \Delta(u,v,w)$ but not of $\Gamma = \Delta(r,s,t)$, contradiction. 
\end{proof} 

(As an illustration, consider the triples $(13,15,117)$ and $(9,39,65)$, which 
satisfy the conclusions of Lemmas~\ref{easyones},~\ref{sameeuler} and~\ref{moreconditions}. 
We can `suppress' $q = 13$ by taking $(u',v',w') = (9,3,5)$, and then find that $\PSL(2,19)$ is 
a quotient of $\Delta(9,39,65)$ but not of $\Delta(13,15,117)$.) 

The observations we have made so far are sufficient to distinguish most triangle groups from 
each other, using just abelian, dihedral and 2-dimensional projective quotients.  
But these are not completely sufficient.  
For example, consider the triples $(15,42,63)$ and $(21,21,90)$, which satisfy the 
conclusions of  Lemmas~\ref{easyones},~\ref{sameeuler} and~\ref{moreconditions}, 
but do not satisfy the hypothesis of Corollary~\ref{coprime} and do not admit the 
kinds of divisors met in Corollary~\ref{nodivisors}.  
For such triples, we need to consider further types of quotients, and it turns out that 
direct products give us almost all we need to complete a proof of the theorem. 
We will use the easily proved fact that if if $G$ and $H$ are finite groups that are 
$(r_1,s_1,t_1)$- and $(r_2,s_2,t_2)$-generated, say by element triples 
$(x_1,y_1,z_1)$ and $(x_2,y_2,z_2)$ respectively, 
and we let $r = \lcm(r_1,r_2)$, $s = \lcm(s_1,s_2)$ and $t = \lcm(t_1,t_2)$, 
then some subgroup of the direct product $G\times H$ is $(r,s,t)$-generated, 
by the triple $((x_1,x_2),(y_1,y_2),(z_1,z_2))$.

\begin{lemma}
\label{q1q2}
If $q_1$ and $q_2$ are coprime positive integers, 
each greater than $3$, such that $q_{1}q_{2}$ divides at least one of $u,v$ and $w$, 
then either $q_{1}q_{2}$ divides at least one of $r,s$ and $t$, 
or otherwise one of $r,s$ and $t$ is prime and equal to $q_1$ or $q_2$.
\end{lemma}

\begin{proof} 
Suppose $q_{1}q_{2}$ divides at least one of $u,v$ and $w$, 
but divides none of $r,s$ and $t$.    
Choose non-trivial divisors $r_1$ and $r_2$ of $r$, 
and non-trivial divisors $s_1$ and $s_2$ of $s$, 
and non-trivial divisors $t_1$ and $t_2$ of $t$, 
as large as possible, such that \\[-18pt]
\begin{itemize}  
\item[{\rm (i)}] $r = \lcm(r_1,r_2)$, $s = \lcm(s_1,s_2)$, and $t = \lcm(t_1,t_2)$, \\[-18pt] 
\item[{\rm (ii)}] each of $r_1,s_1$ and $t_1$ is coprime to $q_1$ to $q_2$, 
with at least one being coprime to $q_1$ and at least one being coprime to $q_2$, 
and  \\[-18pt] 
\item[{\rm (iii)}] no member of the $L_2$-set of $(r_2,s_2,t_2)$  is divisible 
by $q_1$ or $q_2$.  \\[-18pt]
\end{itemize}
It is an easy exercise to see that such a choice is 
always possible.  For example, write $\lcm(r,s,t)$ as $m_{1}m_{2}$, where 
$\gcd(m_1,m_2) = 1$ and $q_1$ divides $m_1$ while $q_2$ divides $m_2$. 
If $q_1$ divides $r$, then take $r_1 = \gcd(r,m_1)$, and take $r_2$ as $\gcd(r,m_2)$
if the latter is positive, or the largest divisor of $r$ not divisible by $q_1$ 
if $r$ divides $m_1$. 
Alternatively, if $q_2$ divides $r$, then take $r_1 = \gcd(r,m_2)$, and 
take $r_2$ as $\gcd(r,m_1)$ if the latter is positive, or the largest divisor 
of $r$ not divisible by $q_2$ if $r$ divides $m_2$.  
If neither $q_1$ nor $q_2$ divides $r$, then take $r_1$ and $r_2$ 
to be $gcd(r,m_1)$ and $\gcd(r,m_2)$ (in either order) if these are both positive, 
or $r_1 = r_2 = r$ if $r$ divides $m_1$ or $m_2$.  
Then make the analogous choices for $s_1$ and $s_2$, and similarly 
for $t_1$ and $t_2$. 
 
With $r_i,s_i$ and $t_i$ chosen this way, it is not difficult to see 
that $(r_1,s_1,t_1)$ is non-exceptional, and more importantly, there exists 
some prime $p$ for which $\PSL(2,p)$ is $(r_1,s_1,t_1)$-generated 
but has no element of order $q_{1}c$ with $\gcd(c,q_2) > 1$ 
or order $q_{2}d$ with $\gcd(d,q_1) > 1$.  
(Hence in particular, $\PSL(2,p)$ has no element of order $q_{1}q_{2}$.) 

Now if the triple $(r_2,s_2,t_2)$ consists of  integers greater than $1$, let $G$ be 
any finite group that is $(r_2,s_2,t_2)$-generated but has no element of 
order divisible by $q_1$ or $q_2$. 
Then some subgroup of the direct product $\PSL(2,p_1) \times G$ 
is $(r,s,t)$-generated, but has no element of order $q_{1}q_{2}$, 
and hence cannot be $(u,v,w)$-generated, contradiction.  
Consequently, at least one of $r_2,s_2$ or $t_2$ is $1$, 
and so (by our choice of $r_i,s_i$ and $t_i$) at least one of $r,s$ and $t$ 
must be equal to $q_1$ or $q_2$, and be prime. 
\end{proof}

We now have enough to prove the main theorem.  

\begin{proof}[{\bf Proof of Theorem~\ref{trianglegroups}}] 
Assume the theorem is false, so that $\Gamma = \Delta(r,s,t)$ and $\Sigma = \Delta(u,v,w)$ 
have the same finite quotients, but $(r,s,t) \ne (u,v,w)$. 

\smallskip
Consider the largest of the six integers $r,s,t,u,v$ and $w$. 
By swapping the roles of $(r,s,t)$ and $(u,v,w)$ if necessary, we may assume 
that $w = \max\{r,s,t,u,v,w\}$. Then by Proposition~\ref{somecases1}(b), 
we know that $w$ is greater than each of $r,s$ and $t$  
(and in particular, $w$ cannot divide $r,s$ or $t$). 
 
If $w$ is a prime-power, then $w$ must divide $\lcm(r,s,t) = \lcm(u,v,w)$ 
and so divides at least one of $r,s$ or $t$, contradiction. 
Thus  $w$ is composite, say $w = q_{1}q_{2}$, with $\gcd(q_1,q_2) = 1$ 
and $1 < q_1 < q_2 < w$.  
Moreover, $q_{1}q_{2} = w$ divides none of $r,s$ and $t$. 
Hence by Lemma~\ref{q1q2} and some elementary number theory, 
we find that one of the following must hold: \\[-15pt]
\begin{itemize}  
\item[{\rm (a)}] $q_1 = 2$ and $q_2$ is an odd prime-power, \\[-18pt] 
\item[{\rm (b)}] $q_1 = 3$ and $q_2$ is a prime-power (not divisible by $3$),  \\[-18pt] 
\item[{\rm (c)}] one of $q_{1}$ and $q_{2}$ is prime and equal to one (or more) of $r,s$ and $t$, 
 and the other is a prime-power, or  \\[-18pt]
\item[{\rm (d)}] $w = 6p$ where $p$ is a prime greater than $3$, and $p$ is equal to one (or more) of $r,s$ and $t$.  \\[-15pt]
\end{itemize}

We will eliminate each of these four cases in turn. 

\medskip
\noindent
{\bf Case (a)}: Here we can write $q_2 = p^k$ where $p$ is an odd prime.  
Then since $p^k$ divides $\lcm(u,v,w) = \lcm(r,s,t)$, we know that $p^k$ divides 
at least one of $r,s$ and $t$, and then since $\max\{r,s,t\} < w = 2p^k$, 
it follows that $p^k$ is equal to at least one of $r,s$ and $t$. 
Similarly, another one of $r,s$ and $t$ is even.
Next, by the argument used in the proof of Lemma~\ref{moreconditions}, 
we can write $\{r,s,t\} = \{2bp^i,cp^j,p^k\}$ 
and $\{u,v,w\} = \{dp^i,ep^j,2p^k\}$, where $b,c,d$ and $e$ are positive integers 
coprime to $p$.  Note also that $bc = de$ since $rst = uvw$, 
and that $i < k$ and $j \le k$ since $\max\{r,s,t\} < w = 2p^k$. \smallskip

Now $rs+rt+st = 2bcp^{i+j}+2bp^{i+k}+cp^{j+k}$ while 
$uv+uw+vw = dep^{i+j}+2dp^{i+k}+2ep^{j+k}$, and since these are equal by 
Lemma~\ref{moreconditions}, we find that $2bc+2bp^{k-j}+cp^{k-i} = de+2dp^{k-j}+2ep^{k-i}$.  
Hence if $j < k$ then $p$ divides $2bc - de = 2bc - bc = bc$, 
which is impossible, and therefore $j = k$. 
In turn, this forces $c = 1$ (because $cp^k = cp^j \le \max\{r,s,t\} < w = 2p^k$) 
and then $e = 2$ (because $p^k$ cannot lie in both $\{r,s,t\}$ and $\{u,v,w\}$, 
by Proposition~\ref{somecases1}(b)), 
but that is impossible, because at most one of $u,v,w$ is even (by Proposition~\ref{somecases1}(c)). 

\medskip
\noindent
{\bf Case (b)}: As in case (a), we can write $q_2 = p^k$ where $p$ is prime, and this time we deduce that 
at least one of $r,s$ and $t$ is equal to $p^k$ or $2p^k$. \smallskip

First, suppose that $p^k \in \{r,s,t\}$.  Then $p \ne 2$, for otherwise the fact that at 
most one of $r,s$ and $t$ is even would imply that $p^k$ is coprime to the other two entries 
of the triple $(r,s,t)$, which is impossible by Corollary~\ref{coprime}.  Thus $p \ge 5$. 
Next, as in case (a), we can write $\{r,s,t\} = \{3bp^i,cp^j,p^k\}$ and $\{u,v,w\} = \{dp^i,ep^j,3p^k\}$, 
where $b,c,d$ and $e$ are positive integers coprime to $p$, with $bc = de$, 
and $i < k$ while $j \le k$.  Equating $rs+rt+st$ with $uv+uw+vw$ and then 
dividing by $p^{i+j}$ gives $3bc+3bp^{k-j}+cp^{k-i} = de+3dp^{k-j}+3ep^{k-i}$. \smallskip

If $j < k$, then $p$ divides $3bc-de = 2bc$, which is impossible,  
so $j = k$.  
This further implies that $c = 1$ or $2$ and $e = 2$ or $3$ and $c \ne e$, 
by Proposition~\ref{somecases1}(b) and the maximality of $w = 3p^k$.  
Also $d = 1$, or else we could divide through by powers of $p$ and apply 
Corollary~\ref{nodivisors}.  Hence in particular, $bc = e$.  
Now if $c = 2$ then $e = 2 = c$, contradiction, so $c = 1$, and $b = e \in \{2,3\}$.
If $b = 2$, however, then comparison of $rs+rt+st$ with $uv+uw+vw$ gives 
$12+p^{k-i} = 5+6p^{k-i}$ and so $5p^{k-i} = 7$, which is impossible, 
and on the other hand, if $b = 3$, then $\lcm(r,s,t) = \lcm(9p^i,p^k,p^k) = 9p^k$ 
while $\lcm(u,v,w) = \lcm(p^i,3p^k,3p^k) = 3p^k$, another contradiction.   \smallskip

Thus $p^k \not\in \{r,s,t\}$, and it follows that one of $r,s,t$ is equal to $2p^k$.  
This time we can write $\{r,s,t\} = \{3bp^i,cp^j,2p^k\}$ and $\{u,v,w\} = \{dp^i,ep^j,3p^k\}$, 
where $b,c,d$ and $e$ are positive integers coprime to $p$, with $2bc = de$, 
and again $i < k$ while $j \le k$.  Equating $rs+rt+st$ with $uv+uw+vw$ and 
dividing by $p^{i+j}$ gives $3bc+6bp^{k-j}+2cp^{k-i} = de+3dp^{k-j}+3ep^{k-i}$, 
and hence if $j < k$ we find that $p$ divides $3bc-de = 3bc-2bc = bc$, contradiction. 
Thus $j = k$.  
Now $c \le 2$ (by maximality of $w = 3p^k$), but on the other hand, 
$c \ne 1$ since $p^k \not\in \{r,s,t\}$,  and $c \ne 2$ since at most one of $r,s,t$ can 
be even, so again we reach a contradiction.  

\medskip
\noindent
{\bf Case (c)}:  Here each $q_i$ is greater than $3$, for otherwise case (a) or (b) applies, 
and so each $q_i$ is coprime to $6$.   

If $q_i$ is prime and and lies in $\{r,s,t\}$, 
then divide each of $u,v,w$ by the highest power of $q_i$ possible, 
to obtain a triple $(u',v',w')$ of integers each of which is coprime to $q_i$. 
Then by Corollary~\ref{nodivisors}, at least one of these integers must be $1$, so 
at least one of $u,v$ and $w$ is a power of $q_i$. 
By Proposition~\ref{somecases1}(b), however, none of them can equal $q_i$,  
so at least one is divisible by ${q_i}^2$.  
On the other hand, none of them is divisible by ${q_2}^2$, 
since ${q_2}^2 > q_{1}q_{2} = w = \max\{r,s,t,u,v,w\}$, and therefore $i = 1$. 
Thus $q_1$ is prime and lies in $\{r,s,t\}$, and at least one of $u,v,w$ is 
divisible by ${q_1}^2$.  \smallskip

Let  $q_1^{\,\gamma}$ be the largest power of $q_1$ that equals one 
or more of $u,v,w$.  Then $\gamma > 1$, and $q_1^{\,\gamma}$ divides one or 
more of $r,s,t$, since $\lcm(r,s,t) = \lcm(u,v,w)$.   
Similarly $q_2$ divides one or more of $r,s,t$, 
but none of $r,s,t$ is divisible by $q_1^{\,\gamma}q_2$ since 
 $q_1^{\,\gamma}q_2 \ge q_{1}q_{2} = w$.
Thus we can write $\{r,s,t\} = \{q_1,bq_1^{\,\gamma},cq_2\}$ 
and $\{u,v,w\} = \{q_1^{\,\gamma},d,q_{1}q_2\}$, 
where $b,c$ and $d$ are positive integers with $b$ coprime to $q_1$,  
and $1 < b <  q_2$ and $c < q_1$ (by Proposition~\ref{somecases1}(b) and maximality of $w$), 
and $d = bc$ (since $rst = uvw$).  
In particular, each of $b,c$ and $d$ is coprime to $q_1$.  \smallskip

Now if $b > 3$, then since $\gamma > 1$ we find that $bq_1^{\,\gamma}$ divides 
none of $\{q_1^{\,\gamma},d,q_{1}q_2\} = \{u,v,w\}$, and then Lemma~\ref{q1q2} gives 
a contradiction.  Hence $b = 2$ or $3$. 
On the other hand, comparing $rs+rt+st$ with $uv+uw+vw$ 
gives  $bq_1^{\,\gamma} + cq_{2} + bcq_1^{\,\gamma-1}q_2 
= dq_1^{\,\gamma-1} + q_1^{\,\gamma}q_2 + dq_2$, 
from which it follows that $q_1$ divides $dq_{2}-cq_{2} = bcq_{2}-cq_{2} = (b-1)cq_{2}$, 
which is impossible (since $q_1 > 3 \ge b$). 

\medskip
\noindent
{\bf Case (d)}: In this final case, we can divide each of $u,v,w$ by the highest power of $p$ possible, 
to obtain a triple $(u',v',w')$ of integers each of which is coprime to $p$. 
By Corollary~\ref{nodivisors}, at least one of $u',v',w'$ must be $1$, so 
at least one of $u,v,w$ is a power of $p$. 
On the other hand, by Proposition~\ref{somecases1}(b), none of $u,v,w$ can equal $p$,  
so at least one of $u,v,w$ is divisible by $p^2$.  
Therefore $p^2$ must divide at least one of $r,s,t$ (since $\lcm(r,s,t) = \lcm(u,v,w)$), 
and again by Proposition~\ref{somecases1}(b), it follows that at least one of $r,s,t,u,v,w$ is 
divisible by $kp^2$ for some $k \ge 2$.
This, however, is impossible because $kp^3 \ge 2p^2 > 6p = w$. 
\end{proof}

We acknowledge the use of {\sc Magma}~\cite{Magma} in helping us find a way to this proof. 
Just as a matter of of interest, we would like to point out that there 
are 3581 pairs of distinct triples $\{(r,s,t), (u,v,w) \}$ with $2 \le r \le s \le t$ and $2 \le u \le v \le w$ 
and $rst = uvw \le$ 12,000,000, 
satisfying the conclusions of Lemmas~\ref{easyones} and~\ref{moreconditions}, 
with at most one of $r,s,t$ being even. 
About half of these 3581 pairs can be eliminated using Corollary~\ref{coprime} 
(the `coprime' test), and then most of the remaining pairs can be ruled out 
by Lemma~\ref{q1q2} (using direct products). 
Just one such small pair cannot be ruled out in this way, 
namely $\{(17,162,459),(27,34,1377)\}$), but this can be eliminated 
by Corollary~\ref{nodivisors} (since $ r= 17$ is coprime to $6, 27, 34/17$ and  $1377/17$).

\bigskip
\noindent  Mathematical Institute,\\
University of Oxford,\\
Andrew Wiles Building,\\
ROQ, Woodstock Road\\ 
Oxford OX2 6GG, U.K.

\noindent Email:~bridson@maths.ox.ac.uk\\

\noindent Department of Mathematics,\\
University of Auckland,\\
Private Bag 92019,\\
Auckland 1142, New Zealand.

\noindent Email:~m.conder@auckland.ac.nz\\

\noindent Department of Mathematics,\\
University of Texas,\\
Austin, TX 78712, U.S.A.

\noindent Email:~areid@math.utexas.edu\\

\end{document}